\documentclass[12pt]{amsart}
\usepackage{amssymb, amsmath, amscd, dcpic, pictexwd}
\usepackage{hyperref}
\usepackage{graphicx}
\usepackage{pdfpages}

\numberwithin{equation}{section}
\newtheorem{thm}{Theorem}[section]
\newtheorem{question}[thm]{Question}
\newtheorem{cor}[thm]{Corollary}
\newtheorem{lem}[thm]{Lemma}
\newtheorem{prop}[thm]{Proposition}
\newtheorem{defn}[thm]{Definition}
\newtheorem{rem}[thm]{Remark}
\newtheorem{exm}[thm]{Example}
\newtheorem{conj}[thm]{Conjecture}

\newcommand{\eq}[2]{\begin{equation}\label{#1}#2 \end{equation}}
\newcommand{\ml}[2]{\begin{multline}\label{#1}#2 \end{multline}}
\newcommand{\ga}[2]{\begin{gather}\label{#1}#2 \end{gather}}

\newcommand{\cD}{\mathcal D}

\newcommand{\cG}{\mathcal{G}}

\newcommand{\cI}{\mathcal I}
\newcommand{\sI}{\mathcal{I}}

\newcommand{\sL}{\mathcal{L}}
\newcommand{\cM}{\mathcal M}

\newcommand{\cO}{\mathcal{O}}
\newcommand{\sO}{\mathcal O}

\newcommand{\cV}{\mathcal V}
\newcommand{\sV}{\mathcal V}

\newcommand{\cY}{\mathcal{Y}}

\newcommand{\R}{{\bf R}}
\renewcommand{\P}{\PP}

\newcommand{\fg}{\mathfrak{g}}

\newcommand{\fl}{\mathfrak{l}}

\newcommand{\fs}{\mathfrak{s}}
\newcommand{\ft}{\mathfrak{t}}

\newcommand{\HH}{\mathbb{H}}

\newcommand{\CC}{\mathbb{C}}

\newcommand{\C}{\CC}
\newcommand{\PP}{\mathbb{P}}
\newcommand{\G}{\mathbb{G}}
\newcommand{\ZZ}{\mathbb Z}
\newcommand{\Z}{\ZZ}

\newcommand{\into}{\hookrightarrow}
\newcommand{\inj}{\into}
\newcommand{\onto}{\twoheadrightarrow}
\newcommand{\surj}{\onto}

\newcommand{\ra}{\rightarrow}
\newcommand{\half}{{1\over 2}}
\newcommand{\bra}{{\langle}}
\newcommand{\ket}{{\rangle}}
\newcommand{\br}{\buildrel}

\newcommand{\blank}{\hskip.3in}
\newcommand{\f}{{\bf f}}

\newcommand{\Proj}{{\mbox{Proj~}}}
\newcommand{\Spec}{{\mbox{Spec~}}}
\newcommand{\Sym}{{\mbox {Sym~}}}
\newcommand{\End}{{\mbox {End~}}}
\newcommand{\Hom}{{\mbox{Hom}}}

\newcommand{\Aut}{{\mbox{Aut~}}}

\newcommand{\Res}{{\mbox{Res~}}}

\newcommand{\Der}{{\mbox{Der~}}}

\def\comment#1{{}}
\def\question#1{{}}

\author{Spencer Bloch, An Huang, Bong H. Lian, Vasudevan Srinivas, and Shing-Tung Yau}

\title{On the Holonomic Rank Problem}

\begin{document}

\maketitle
\begin{abstract}
A tautological system, introduced in \cite{LSY}\cite{LY}, arises as a regular holonomic system of partial differential equations that govern the period integrals of a family of complete intersections in a complex manifold $X$, equipped with a suitable Lie group action. 
In this article, we introduce two formulas -- one purely algebraic, the other geometric -- to compute the rank of the solution sheaf of such a system for CY hypersurfaces in a generalized flag variety. The algebraic version gives the local solution space as a Lie algebra homology group, while the geometric one as the middle de Rham cohomology of the complement of a hyperplane section in $X$. We use both formulas to find certain degenerate points for which the rank of the solution sheaf becomes 1. These rank 1 points appear to be good candidates for the so-called large complex structure limits in mirror symmetry. The formulas are also used to prove a conjecture of Hosono, Lian and Yau on the completeness of the extended GKZ system when $X$ is $\P^n$. 
\end{abstract}

\tableofcontents 
\baselineskip=16pt plus 1pt minus 1pt
\parskip=\baselineskip

\pagenumbering{arabic}
\addtocounter{page}{0}
\markboth{\SMALL 
Spencer Bloch, An Huang, Bong H. Lian, Vasudevan Srinivas, and Shing-Tung Yau}
{\SMALL On the Holonomic Rank Problem}

\section{Introduction}

Let $X$ be a compact complex manifold, such that the complete linear system of anti-canonical divisors in $X$ is base point free. In \cite{LY}, the period integrals of the corresponding universal family of CY hypersurfaces is studied. It is shown that they satisfy a certain system of partial differential equations defined on the affine space $V^*=\Gamma(X,\omega_X^{-1})$ which we call a {\it tautological system}. 
When $X$ is a homogeneous manifold of a semi-simple Lie group $G$, such a system can be explicitly described. For example, one description says that the tautological system can be generated by the vector fields corresponding to the linear $G$ action on $V^*$, together with a set of quadratic differential operators corresponding to the defining relations of $X$ in $\P V$ under the Pl\"ucker embedding. The case where $X$ is a Grassmannian has been worked out in detail \cite{LSY}.

\begin{defn} \cite{LSY}\cite{LY}
Let $\hat G$ be complex Lie group, $Z:\hat G\ra\Aut V$ be a given holomorphic representation such that $Z(\hat G)$ contains $\C^\times 1_V$, and let $Z:\hat\fg\ra\End V$ be the corresponding Lie algebra representation. Let $\hat X\subset V$ be a $\hat G$-stable subvariety, and $\beta:\hat\fg\ra\C$ be a Lie algebra homomorphism, 
The tautological system $\tau(\hat X,V,\hat G,\beta)$ is the differential system generated by the operators
\begin{eqnarray*}
&Z_\beta(x):=Z(x)+\beta(x),~~~~x\in\hat\fg\cr
&p(\partial_\zeta),~~~~p(\zeta)\in I(\hat X,V).
\end{eqnarray*}
Here $\partial_\zeta\in\Der\C[V^*]$ is defined by $\partial_\zeta\cdot a=\bra\zeta,a\ket$ ($a\in V$, $\zeta\in V^*$); $I(\hat X,V)\subset\C[V]$ is the defining ideal of $\hat X\subset V$.
\end{defn}

Note that in the definition, we can view $Z(x)\in\End V$ as a differential operator on $V^*$ because $\End V\subset\Der(\Sym V)=\Der\C[V^*]$.

There are a number of important special cases of this definition that have been extensively studied in various context. For a brief overview of these special cases, see \cite{LSY}.

Let $\pi:\cY\ra B:=\Gamma(X,\omega_X^{-1})_{sm}$ be the family of smooth CY hyperplane sections in $X$, and let $\HH^{top}$ be the Hodge bundle over $B$ whose fiber at $f\in B$ is the line $\Gamma(Y_f,\omega_{Y_f})\subset H^{n-1}(Y_f)$, where $n=\dim X$. In \cite{LY}, the period integrals of this family are constructed by giving a canonical trivialization of $\HH^{top}$. Let $\Pi=\Pi(X)$ be the period sheaf of this family, i.e. the locally constant sheaf generated by the period integrals (Definition 1.1 \cite{LY}.)

\begin{thm} \label{LY-theorem}
The period integrals of the family are solutions to the tautological system $\cM=\tau(\hat X,V,\hat G,(0;1))$, where $\hat X$ is the cone over $X$ in $V=\Gamma(X,\omega_X^{-1})^*$, and $\hat G=\Aut X\times\G_m$.
\end{thm}
This was proved in \cite{LSY} for $X$ a partial flag variety, and in full generality in \cite{LY}, in which the result was also generalized to hyperplane sections of general type. 
Applying an argument of \cite{Kapranov1997}, it was also shown that if $X$ has only a finite number of $G=\Aut X$-orbits, then $\cM$ is regular holonomic (Theorem 3.4 \cite{LSY}.) In this case, if $X=\sqcup_{l=1}^rX_l$ is the decomposition into $G$-orbits, then the singular locus of $\cM$ is contained in $\cup_{l=1}^rX_l^\vee$. Here $X_l^\vee\subset V^*$ is the conical variety whose projectivization $\PP(X_l^\vee)$ is the projective dual to the Zariski closure of $X_l$ in $X$. 


In the well-known applications of variation of Hodge structures in mirror symmetry, it is important to decide which solutions of our differential system come from period integrals. By Theorem \ref{LY-theorem}, the period sheaf is a subsheaf of the solution sheaf of a tautological system. Thus an important problem is to decide when the two sheaves actually coincide. If they do not coincide, how much larger is the solution sheaf? From Hodge theory, we know that  (see Proposition \ref{period-rank}) the rank of the period sheaf is given by the dimension of the middle vanishing cohomology of the smooth hypersurfaces $Y_f$.
Therefore, to answer those questions, it is desirable to know precisely the holonomic rank of our tautological system.

Let us recall what is known on these questions in a number of special cases.
In the case of CY hypersurfaces in, say, a semipositive toric manifold $X$, it is known \cite{GKZ1990}\cite{Adolphson} that the holonomic rank of the GKZ hypergeometric system in this case is the normalized volume of the polytope generated by the exponents of the monomial sections in $\Gamma(X,\omega_X^{-1})$. This number is also the same as the degree of the anticanonical embedding $X\into\P\Gamma(X,\omega_X^{-1})^*$. However, it is also known \cite{HLY1994} that this number always exceeds (and is usually a lot larger than) the rank of the period sheaf. If one considers the extended GKZ hypergeometric system, where the torus $T$ acting on $X$ is replaced by the full automorphism group $\Aut X$, one would expect that the rank of the extended system to be closer to that of the period sheaf. In fact, based on numerical evidence, it was conjectured \cite{HLY1994} that for $X=\P^n$ (which lives in both the toric world and the homogeneous world), the rank of $\cM$ coincides with that of the period sheaf at generic points. In the case when $X=X_A$ is a spherical variety of a reductive group $G$ corresponding to a given set of irreducible $G$-modules $A$, Kapranov \cite{Kapranov1997} showed that the rank of his A-hypergeometric system is bounded above by the degree of embedding $X_A\subset\PP M_A^*$, if the cone $\hat X_A$ over $X_A$ in $M_A^*$ is assumed to be Cohen-Macaulay. This result was generalized to any smooth $G$-variety $X$ with a finite number of $G$-orbits by Lian, Song and Yau \cite{LSY}.  Note, however, that the rank upper bound in each case cited above makes no assumptions about whether the underlying D-module arises from the variation of Hodge structures of CY varieties. Moreover, since the holonomic rank gives the number of independent solutions only away from the singular locus, the bound yields no information about solutions at singularities.

In this paper, we introduce two new formulas -- one purely algebraic, and the other geometric -- to compute the rank of the solution sheaf of a tautological system for CY hyperplanes sections in $X$. The algebraic formula expresses the solution space at any given point (singular or not), as the dual of a certain Lie algebra homology with coefficient in the coordinate ring of $X$ (Theorem \ref{Lie-algebra-homology}.) The geometric formula uses the algebraic result to identify the solution space with the middle de Rham cohomology of the complement of the same CY hyperplane section in $X$.  Based on much numerical evidence, it is conjectured that the geometric result holds for an arbitrary homogeneous variety. Our proof is valid for most familiar cases (e.g. projective spaces, Grassmannians, quadrics, spinor varieties, maximal Lagrangian Grassmannians, two exceptional varieties, full flag varieties $G/B$, and products of such). 

We also use both formulas to find certain degenerate points for which the rank of the solution sheaf is 1. We conjecture that the rank 1 points in Theorem \ref{rank1-points} in fact correspond to large complex structure limits (in the sense of \cite{Morrison}\cite{Gross}), in the moduli space of CY hypersurfaces in $X$.

\begin{conj} (Holonomic rank conjecture) \label{holo-rank}
Let $X$ be an $n$-dimensional projective homogeneous space of a semisimple Lie group $G$. Then the dimension of the solution space of the tautological system $\tau(\hat X,V,G\times\G_m,\beta)$, where $V=\Gamma(X,\omega_X^{-1})^*,~\beta=(0;1)$, at the point $f\in V^*$, coincides with
$$
\dim H^n_{dR}(X-Y_f).
$$
\end{conj}

In this paper, we will prove 
\begin{thm}\label{geo.thm}
Assume that the natural map
$$
\fg\otimes\Gamma(X,\omega_X^{-r})\ra\Gamma(X,T_X\otimes\omega_X^{-r})
$$
is surjective for each $r\geq0$. Then conjecture \ref{holo-rank} holds for all $f\in V^*$. 
\end{thm}
The list of homogeneous spaces known to satisfy the condition in the theorem includes Grassmannians, full flag varieties $G/B$, quadrics, spinor varieties, maximal Lagrangian Grassmannians, and two exceptional $X$'s, as well as products of such see Proposition \ref{prop4b}. 
As one immediate consequence, we also deduce the following

\begin{cor} \label{HLY-conjecture}\cite{HLY1994}
For $X=\P^n$, the tautological system 
$$\tau(\hat X,\Gamma(X,\omega_X^{-1})^*,SL_{n+1}\times\G_m,(0;1))$$ 
(which is a special case of an extended GKZ system) is complete. In other words, the solutions at a generic point $f\in V^*$ are precisely the period integrals of CY hypersurfaces in $X$. 
\end{cor}

More generally, we will show as a corollary of Conjecture \ref{holo-rank} that the question of completeness for a given $X$ can be reduced to the vanishing of the middle primitive cohomology of $X$, which is essentially topological.

\section{Solution sheaf and Lie algebra homology}

We begin with the set up in \cite{LY} and consider the rank of the solution sheaf to the tautological system $\tau(\hat X,V,\hat G,\beta)$. 
We have a holomorphic representation
$$
Z:\hat\fg\ra \End V
$$
and its contragredient dual representation 
$$
Z^*:\hat\fg\ra \End V^*.
$$

Since $\End V\subset \Der(\Sym V)=\Der \C[V^*]$ and $\End V^*\subset\Der(\Sym V^*)=\Der\C[V]$, we can view for $x\in\hat\fg$,
$$
Z(x)\in\Der\C[V^*],~~~~Z^*(x)\in\Der\C[V].
$$
Thus by fixing a basis $a_i$ for $V$ and dual basis $a^*_i$ for $V^*$, we can write 
$$
\C[V^*]=\C[a],~~~\C[V]=\C[a^*]
$$
and
$$
Z(x)=\sum_{i,j}x_{ji}a_j{\partial\over\partial a_i},~~~Z^*(x)=-\sum_{i,j}x_{ij}a^*_j{\partial\over\partial a^*_i}.
$$
Put
$$
Z_\beta(x)=Z(x)+\beta(x) ~~~(x\in\hat\fg.)
$$

\begin{defn} \label{redefine-D-module}
Let $\cD=\CC[a][\partial_1,\partial_2,...]$ be the Weyl algebra on $V^*$, where $\partial_i={\partial\over\partial a_i}$, and consider the linear  isomorphism 
$$
\Phi:\cD\ra \CC[V\times V^*]= \C[a,a^*], ~~~\sum_u g_u(a)\partial^u\mapsto\sum_u g_u(a){a^*}^u.
$$
\comment{We can view $\Phi$ as a map from $\cD$ to the associated graded $gr(\cD)$ under the order filtration.}
Let $\Psi:\cD\ra\End\C[a,a^*]$ be the $\cD$-module structure induced by $\Phi$, i.e.
$$
\Psi(Q)\cdot q=\Phi(Q\cdot\Phi^{-1}(q)),~~~(Q\in\cD,~q\in\C[a,a^*].)
$$
\end{defn}

Observe that on the variables $\partial_i$, $\Phi$ is precisely the inverse of the Fourier transform we used to define the tautological system $\tau(\hat X,V,\hat G,\beta)$ in \cite{LSY}\cite{LY}. Cf. Eqn. (4.5) \cite{Adolphson}.

Next, it is straightforward to check

\begin{lem}\label{Psi}
We have $\Psi(a_i)=a_i$ (acting by left multiplication), and $\Psi(\partial_i)={\partial\over\partial a_i}+a^*_i$.
Let $I=I(\hat X,V)\subset\C[a^*]=\C[V]$ be the vanishing ideal of $X$ in $\PP V$. Then the ideal $\C[a]I$ of the ring $\CC[a,a^*]$ is a $\cD$-submodule of $\CC[a,a^*]$ under the action $\Psi$.
\end{lem}

\begin{lem} \label{1st-sub}
The image under $\Phi$ of $\cD\Phi^{-1}(I)$ is $\CC[a]I$. In particular, $\Phi$ induces a $\cD$-module isomorphism
$$
\cD/\cD\Phi^{-1}(I)\cong R[a]
$$
where $R=R_V:=\CC[V]/I(\hat X,V)=\CC[a^*]/I$.
\end{lem}
\begin{proof}
Since $\Phi^{-1}(I)\subset\C[\partial_1,\partial_2,..]$, we have $\Phi(\cD\Phi^{-1}(I))=\CC[a,a^*]I$.
The right side is $\C[a]I$, since $I$ is an ideal in $\C[a^*]$.
\end{proof}

Put 
$$\f=\sum_i a_ia^*_i\in\CC[V\times V^*]=\CC[a,a^*]$$
which is the ``generic'' hyperplane section under the embedding $X\subset\PP V$. 
Then by a straightforward calculation, we find that

\begin{lem}\label{Z-star-f-beta}
The map $Z^*_{\f,\beta}:\hat\fg\ra\End \CC[a,a^*]$ given by
$$
x\mapsto Z^*_{\f,\beta}(x)=Z^*(x)+(Z^*(x)\f)-\beta(x)~~(x\in\hat\fg)
$$
is a Lie algebra homomorphism. (Here $(Z^*(x)\f)$ means the operator ``multiplication by $Z^*(x)\f$.'' This is not the same as the composition of the two operators $Z^*(x)$ and multiplication by $\f$.) 
\end{lem}

As we shall see later, in the case when $\hat\fg$ is a direct sum of Lie algebras $\fg\oplus\C$, the choice $\beta=(0;1)$ and $Z^*(1)$ being the negative Euler operator on $\CC[a^*]$ will be important for computing the holonomic rank using the method of Feynman measure. Note further that the lemma also holds true if we replace $\f$ by a {\it fixed} section $f=\sum_ia^{(0)}_ia^*_i\in V^*$ and $\CC[a,a^*]$ by $\CC[a^*]$ (i.e. evaluate the $a_i$ at $a_i=a^{(0)}_i\in\C$), since the derivations $Z^*(x)\in\Der\C[a^*]$ do not affect the variables $a_i$ in the calculation leading to Lemma \ref{Z-star-f-beta}.

\begin{lem}\label{g-D-commute}
For $x\in\hat\fg$, $Z^*_{\f,\beta}(x)\in\End_\cD\CC[a,a^*]$. In other words, the $\hat\fg$-action and the $\cD$-action on $\C[a,a^*]$ commute.
\end{lem}
\begin{proof}
It is obvious that $[Z^*_{\f,\beta}(x),a_i]=0$. We also have
$$
[Z^*_{\f,\beta}(x),{\partial\over\partial a_i}+a^*_i]=0.
$$
By Lemma \ref{Psi}, it follows that $[Z^*_{\f,\beta}(x),\Psi(\cD)]=0$, i.e. the $\hat\fg$-action and the $\cD$-action on $\CC[a,a^*]$ commute.
\end{proof}

\begin{lem} \label{action-on-Ra}
The $\cD$-submodule $\C[a]I\subset\C[a,a^*]$ is also a $\hat\fg$-submodule, where $\hat\fg$ acts via the operators $Z^*_{\f,\beta}(x)$. Hence $\hat\fg$ acts on the quotient $R[a]=\C[a,a^*]/\C[a]I$.
\end{lem}
\begin{proof}
Since $I\subset\C[a^*]=\C[V]$ is the vanishing ideal of the $\hat G$ invariant subvariety $\hat X\subset V$, and since the Lie algebra $\hat\fg$ of $\hat G$ acts on $\C[V]$ by $Z^*:\hat\fg\ra\Der \C[V]$, it follows that $Z^*(x) I\subset I$. Since the $Z^*(x)\f$ acts on $\C[V\times V^*]=\C[a,a^*]$ by left multiplication, they also leave $\C[a]I$ stable. It follows that the $Z^*_{\f,\beta}(x)=Z^*(x)+Z^*(x)\f-\beta(x)$ leave $\C[a]I$ stable.
\end{proof}

\begin{lem} \label{2nd-sub}
For $x\in\hat\fg$, $\Phi Z_\beta(x)=-Z^*_{\f,\beta}(x)\cdot 1$. Moreover, the image under $\Phi$ of $\cD Z_\beta(\hat\fg)$ is $Z^*_{\f,\beta}(\hat \fg)\cdot\C[a,a^*]$.
\end{lem}
\begin{proof}
For $x\in\hat\fg$, we have 
\begin{eqnarray*}
\Phi Z_\beta(x)&=&\Phi(\sum x_{ji}a_j{\partial\over\partial a_i}+\beta(x))\cr
&=&\sum x_{ji}a_ja^*_i+\beta(x)\cr
&=&-Z^*(x) \f+\beta(x)=-Z^*_{\f,\beta}(x)\cdot 1
\end{eqnarray*}
which gives the first assertion. Since the $\cD$-action on $\C[a,a^*]$ commutes with the $Z^*_{\f,\beta}(x)$ by Lemma \ref{g-D-commute}, it follows that
$$
\Phi(\cD Z_\beta(x))\subset Z^*_{\f,\beta}(x)\C[a,a^*].
$$
Hence $\Phi(\cD Z_\beta(\hat\fg))\subset Z^*_{\f,\beta}(\hat\fg)\C[a,a^*]$.
To see the reverse inclusion, let $q\in\C[a,a^*]$, $x\in\hat\fg$. We have
\begin{eqnarray*}
Z^*_{\f,\beta}(x)q &=&Z^*_{\f,\beta}(x)\Phi(\Phi^{-1}(q)\cdot 1)\cr
&=&Z^*_{\f,\beta}(x)\Psi(\Phi^{-1}(q))\cdot1\cr
&=&\Psi(\Phi^{-1}(q))Z^*_{\f,\beta}(x)\cdot1\cr
&=&-\Psi(\Phi^{-1}(q))\Phi(Z_\beta(x))\cr
&=&-\Phi(\Phi^{-1}(q)Z_\beta(x))\in \Phi(\cD Z_\beta(x)).
\end{eqnarray*}
Here the second, fourth and last equalities follow from Definition \ref{redefine-D-module}, while the third equality follows from Lemma \ref{g-D-commute}. This proves the reverse inclusion.
\end{proof}

Recall Definition 8.1 \cite{LY} 
$$
\tau(\hat X,V,\hat G,\beta):=\cD/(\cD Z_\beta(\hat\fg)+\cD\Phi^{-1}I(\hat X,V))
$$
Combining Lemmas \ref{1st-sub}, \ref{action-on-Ra} and \ref{2nd-sub}, we get

\begin{thm}
$\Phi$ induces a $\cD$-module isomorphism
$$
\tau(\hat X,V,\hat G,\beta)\cong R[a]/Z^*_{\f,\beta}(\hat\fg)R[a]
$$
where $R:=\C[V]/I(\hat X,V)$ and $R[a]=R[V^*]$.
\end{thm}

For $a^{(0)}\in\C^{\dim V}$, let $\hat\cO_{a^{(0)}}$ be the $\cD$-module of formal power series at $a^{(0)}$, and $\cO_{a^{(0)}}$ the $\cD$-module of convergent power series at $a^{(0)}$. Let $\C_{a^{(0)}}=\C$ be the one dimensional $\C[a]$-module such that $a_i$ acts by $a^{(0)}_i$. As before, we put 
$$
f=\sum_i a_i^{(0)} a_i^*\in V^*.
$$

\begin{thm}\label{Lie-algebra-homology}
Suppose $\hat X$ has only a finite number of $\hat G$-orbits. Put $\cM=\tau(\hat X,V,\hat G,\beta)$. Then  we have
$$
\Hom_\cD(\cM,\cO_{a^{(0)}})\cong \Hom_\cD(\cM,\hat\cO_{a^{(0)}}) \cong H^{Lie}_0(\hat\fg,R_f)^*
$$
where $\hat\fg$ acts on $R_f:=\C[\hat X]$ by 
\begin{eqnarray*}
& &Z^*_{f,\beta}:\hat\fg\ra\End R_f\cr 
& &x\mapsto Z^*(x)+Z^*(x)f-\beta(x).
\end{eqnarray*}
\end{thm}

Part of the argument of Theorem 4.17 \cite{Adolphson} generalizes to our setting. The main point here is that even though the argument there which contained calculations that relied heavily on the assumptions that $X$ is a toric variety and that the ideal $I$ is binomial, when the argument is reinterpreted suitably, the assumptions turn out to be unnecessary. One further new observation here is that it is useful to interpret the space $R_f/Z^*_{f,\beta}(\hat\fg)R_f$ as the Lie algebra homology of the $\hat\fg$-module $R_f$.

\begin{proof}
By Theorem 3.4 \cite{LSY}, $\cM$ is regular holonomic, so the first isomorphism holds (Proposition 14.8 \cite{Borel}.) 
Since $\cM$ is a finitely generated $\cD$-module, we have
$$
\Hom_\cD(\cM,\hat\cO_{a^{(0)}})\cong\Hom_\C(\C_{a^{(0)}}\otimes_{\C[a]}\cM,\C).
$$
By the preceding theorem, the right side is 
\begin{eqnarray*}
\Hom_\C(\C_{a^{(0)}}\otimes_{\C[a]} R[a]/Z^*_{\f,\beta}(\hat\fg)R[a],\C)&\cong& 
\Hom_\C(R_f/Z^*_{f,\beta}(\hat\fg)R_f,\C)\cr
&\cong& H^{Lie}_0(\hat\fg,R_f)^*.
\end{eqnarray*}
This completes the proof.
\end{proof}

\begin{rem}\label{Ref}\hskip1in\\
\noindent (a) Consider the linear isomorphism $R_f\ra R e^f$, $\phi\mapsto\phi e^f$. Under this identification, the $\hat\fg$-action by $Z^*_{f,\beta}(\hat\fg)$ on $R_f$ corresponds to the action 
$$
\hat\fg\otimes Re^f\ra Re^f,~~~x\otimes \phi e^f\mapsto (Z^*(x)-\beta(x))(\phi e^f).
$$
From now on, $H^{Lie}_*(\hat\fg,Re^f)$ will be understood to be the Lie algebra homology with respect to this action. We can do the same for the $\hat\fg$-modules $I(\hat X,V)$ and $\C[V]$. \\
\noindent (b) We will see that writing the $\hat\fg$-action as such allows us to use the idea of Feynman measures \cite{Borcherds} to directly compute the holonomic rank of $\tau(\hat X,V,\hat G,\beta)$ in some cases.\\
\noindent (c) In a later section, we will reinterpret the Lie algebra homology in the theorem in terms of certain de Rham cohomology, in the case $X\into\P V$ where $V=\Gamma(X,\omega_X^{-1})^*$, $\hat G=G\times\G_m$ where $G$ is semisimple and $\beta=(0;1)$. 
\end{rem}
\begin{exm}
As an application of Theorem \ref{Lie-algebra-homology}, we will show that for $X=\P^n$ the period integrals of Calabi-Yau hypersurfaces form a complete set of solutions to the tautological system  (Corollary \ref{HLY-conjecture}). This will also turn out to be an easy consequence the geometric formula (Theorem \ref{geo.thm}) later. 
\end{exm}

We will eventually specialize to the Fermat case $f=x_0^{n+1}+\cdots +x_n^{n+1}$, but for now $f$ can be any smooth CY hyperplane section. First consider the period sheaf, i.e. the sheaf defined on $\Gamma(X,\omega_X^{-1})_{sm}$, that is generated by the period integrals of smooth CY hypersurfaces $Y_f$ in $X$. By Proposition \ref{period-rank}, it is locally constant of rank given by the dimension of the vanishing cohomology which is
$$
\nu_n:=\dim H^{n-1}(Y_f)-\dim i^* H^{n-1}(X)
$$
where $i:Y_f\into X$ is the inclusion map. By the Lefschetz hyperplane theorem,  it is easy to show that for $X=\P^n$,
$$
\nu_n=\frac{n}{n+1}(n^n-(-1)^n).
$$
Since the period sheaf is a subsheaf of the solution sheaf $Sol(\cM)$ (Theorem \ref{LY-theorem}),  $\dim H^{Lie}_0(\hat\fg, Re^f)\geq\nu_n$. Thus it remains to show that
\eq{to-prove}{\dim H^{Lie}_0(\hat\fg, Re^f)\leq \nu_n.} 
Note that we can identify $R$ with the subring of $\C[x]:=\C[x_0,...x_n]$ consisting of polynomials of degrees divisible by $n+1$.

\begin{lem}
We have
\eq{FM}{
\hat\fg\cdot(Re^f)=Re^f\cap\sum_i\frac{\partial}{\partial x_i}(\C[x]e^f).}
\end{lem}
\begin{proof}
Consider the action $\hat\fg\ra\End Re^f$, $y\mapsto Z^*(y)-\beta(y)$. Here $Z^*$ comes from the dual of the representation
$Z: \hat\fg=\fg\fl_{n+1}\ra\End V$. We have 
$$
Z^*(1)=-{1\over n+1}\sum_i x_i{\partial\over\partial x_i}=-{1\over n+1}\sum_i{\partial\over\partial x_i}x_i+1
$$
We also have $Z^*(X_{ij})=-x_i{\partial\over\partial x_j}=-{\partial\over\partial x_j}x_i$ ($i\neq j$), and $Z^*(H_i)=-x_0{\partial\over\partial x_0}+x_i{\partial\over\partial x_i}=-{\partial\over\partial x_0}x_0+{\partial\over\partial x_i}x_i$, where the $X_{ij},H_i$ form the standard basis of $\fg=\fs\fl_{n+1}$. Using $\beta(\fg)=0$ and $\beta(1)=1$ (this value is crucial!), it follows easily that
$$
(Z^*-\beta)(\hat\fg)=\sum_{ij}\C{\partial\over\partial x_i}x_j.
$$
This shows that the left side of \eqref{FM} is a subspace of the right side. 

To see the reverse inclusion, we consider the $\mathbb{Z}/(n+1)\mathbb{Z}$ grading on $\C[x]e^f$: for $p(x)\in\C[x]$ a degree $k$ polynomial, the grading of $p(x)e^f$ is $k\mod(n+1)$. Since $R\subset\C[x]$ is the subring generated by polynomials of degree $0\mod(n+1)$, both sides of \eqref{FM} have grading $0\mod(n+1)$. Let $A$ be an element on the right side of \eqref{FM}, so that it has grading $0\mod(n+1)$ and it has the form
$$
A=\sum_i\frac{\partial}{\partial x_i}(p_i(x) e^f)
$$
where $p_i(x)\in\C[x]$. By grouping homogeneous terms, we may as well assume that the $p_i(x)$ have polynomial degree $1\mod(n+1)$, which means that $p_i(x)=\sum_j x_j p_{ji}(x)$ for some $p_{ji}(x)$ of degree $0\mod(n+1)$. This shows that $A=\sum_{ij}\frac{\partial}{\partial x_i}x_j(p_{ji}(x) e^f)$ which lies in the left side of \eqref{FM}. This proves the reverse inclusion.
\end{proof}

To complete the proof of the Corollary \ref{HLY-conjecture}, we now choose
$$
f=x_0^{n+1}+\cdots+x_n^{n+1}.
$$
Consider an element of the form $x_0^{k_0}...x_n^{k_n}e^f$ in $Re^f$ with $k_0\geq n$. Then
$$
\frac{\partial}{\partial x_0}(x_0^{k_0-n}x_1^{k_1}...x_n^{k_n}e^f)=(n+1)x_0^{k_0}...x_n^{k_n}e^f+\frac{\partial}{\partial x_0}(x_0^{k_0-n}x_1^{k_1}...x_n^{k_n})~e^f.
$$
By the lemma, $x_0^{k_0}...x_n^{k_n}e^f$ and $\frac{\partial}{\partial x_0}(x_0^{k_0-n}x_1^{k_1}...x_n^{k_n})~e^f$ represent the same element in $H^{Lie}_0(\hat\fg,Re^f)$. The analogous statement also holds for each $k_i\geq n$. It follows that any element in $H^{Lie}_0(\hat\fg,Re^f)$ can be represented by a linear combination of elements of the form $x_0^{k_0}...x_n^{k_n}e^f$, where the $k_i$ are at most $n-1$, and
$$
k_0+\cdots+ k_n\equiv 0\mod(n+1).
$$

For integer $0\leq s\leq (n+1)(n-1)$, we denote by $a(s)$ the number of integer solutions to the equation
$$
k_0+\cdots+k_n=s
$$
where the $k_i$ are between $0$ and $n-1$. By the preceding paragraph
\eq{1}{
\dim H^{Lie}_0(\hat\fg,Re^f)\leq \sum_{(n+1)|s}a(s)}
To calculate the right side, we consider the Gauss sum
$$
\sum_{\lambda=0}^n\sum_{0\leq k_i\leq n-1}e^{\frac{2\pi i(k_0+...+k_n)\lambda}{n+1}}
$$
which is equal to $(n+1)\sum_{(n+1)|s}a(s)$. On the other hand, by summing over the $k_i$ individually, this sum is equal to $\sum_{\lambda=0}^{n}(\frac{1-\xi^{n\lambda}}{1-\xi^{\lambda}})^{n+1}$, where $\xi=e^{\frac{2\pi i}{n+1}}$. The latter sum can be calculated, and it is equal to $(n+1)\nu_n$. Therefore, 
\eq{2}{\sum_{(n+1)|s}a(s)=\nu_n.}
Finally \eqref{1} and \eqref{2} yield \eqref{to-prove}.

\section{Lie algebra homology in geometric terms}\label{LA-geometry}

Let $X$ be a smooth, projective variety over $\C$, and let $\pi_L: L \to X$ be a line bundle over $X$. Let $G$ be a 
reductive Lie group acting on $(L,X)$. Let $\frak g$ be the Lie algebra of $G$. $\frak g$ acts by derivations on
the function algbera $\Sym L^*$. In particular, if $\ell$ is a local section of $L^*$ and $f$ a function on $X$, then for 
$x\in \frak g$ we have $x(f\ell) = fx(\ell) + x(f)\ell$; i.e. $\frak g$ acts by relative derivations. In addition, the action is linear, i.e. $\frak g$ preserves the grading on $\Sym L^*$. \comment{We can think of $\ell$ as a (local) function on $L$ which is linear along fibers, and $f$ a local function on $L$ which is constant along the fiber.}

Let $U := L-\{0\}$ be the complement of the zero section, and write $\pi= \pi_L|U:U \to X$. We have an exact sequence 
of tangent bundles
\eq{1b}{0 \to \pi_L^*L \to T_L \to \pi_L^*T_X \to 0. 
}
Here $\pi_L^*L$ is the sheaf of tangent vectors along the fibres. Restricting \eqref{1b} over $U$ yields
\eq{2b}{0 \to \sO_U \to T_U \to \pi^*T_X \to 0.
}
The sheaf $\sO_U$ in this context has a canonical generator which is the Euler operator $E$ given by 
$E(f\ell^r) = rf\ell^r$.  \comment{Here, we can think of $E\in\Gamma(U,T_U)$, $\ell$ as a (local) function on $U\subset L$ which is linear along fibers, and $f$ a local function on $U$ which is constant along the fiber.} Assuming the action of $\frak g$ is faithful, we have
\eq{}{\frak g + \C\cdot E \subset \Gamma(U, T_U)
}
Another way to think about this is to note that $G\times \G_m$ acts on $L$. It follows that $E$ 
commutes with the action of $\frak g$. Assuming that $\frak g$ acts faithfully on $X$ (note even in the homogeneous space
case, this is an extra hypothesis if $\frak g$ is not simple) 
we will have a 
sub-Lie algebra $\frak g\oplus \C\cdot E \inj \Gamma(U,T_U)$.

Because of the $\G_m$-action on $U$, sheaves like $\pi_*\sO_U, \pi_*T_U, \text{ and } \pi_*\Omega^i_U$ will have a grading. 
The exterior derivative 
$d: \pi_*\Omega^i_U \to \pi_*\Omega^{i+1}_U$ has degree $0$. 
\comment{Cf. next lemma with Lemma 3.2 p16 [LY].}
 
\begin{lem}\label{unique-top-form}
$\pi^*L$ has a unique (upto $\C^\times$) non-vanishing section of degree $-1$ for the $\G_m$-action. 
\end{lem}
\begin{proof}
 It suffices to check with $X=\P^n= \Proj\C[t_0,\dotsc,t_n],\ L=\sO_{\P^n}(1)$. Write $\P^n=\bigcup U_i$ with 
$U_i = \Spec \C[t_0/t_i,\dotsc,t_n/t_i]$. Identify $L|U_i = \sO_{U_i}$ with transition cocycle 
$\sigma_{ij} = t_j/t_i$ (so the sections $t_k/t_i$ on $U_i$ glue to global sections). We have 
$\pi^{-1}(U_i) = \Spec \C[t_0,\dotsc,t_n,t_i^{-1}]$The global section of $\pi^*L$ is given by $t_i^{-1}$ on 
$\pi^{-1}(U_i)$. 
\end{proof}

Assume now that the canonical bundle $\omega_X \cong L^{-N}$ for some $N\ge 1$. We have 
$\Omega^1_{U/X} = \sO_U\frac{d\ell}{\ell}$, so there is an isomorphism which we denote by $\alpha$
\eq{4b}{\alpha: \omega_U \cong \pi^*\omega_X \cong \sO_U[-N]
}
\comment{The first isomorphism follows from adjunction for the principal bundle $\C^\times - U\ra X$. The second isomorphism follows from that $\omega_X=L^{-N}$ and that $\C^\times$ acts on $L-0$ by degree 1 scaling.}
Note that if $\alpha_1, \alpha_2$ are two choices for such an isomorphism, then $\alpha_2\circ\alpha_1^{-1}$ is an 
isomorphism $\sO_U \to \sO_U$ of degree $0$, hence it lies in $\C^\times$. In particular, if we assume the group
$G$ is semi-simple and hence has no abelian characters, the isomorphism $\alpha$ is invariant under the action of $G$. 

Let $\dim X = n$, so $\omega_U = \Omega^{n+1}_U$. Then we define a map $\theta$
\eq{5b}{\theta: T_U = Hom(\Omega^1_U, \sO_U) \cong Hom(\Omega^1_U, \omega_U)[N] = \Omega^n_U[N].
}
\comment{The last equality is the identification under the natural isomorphism $\wedge^n F\cong Hom(F,\wedge^{n+1}F)$, $u\mapsto(F\ra\wedge^{n+1}F,~v\mapsto u\wedge v)$, with $dim~F=n+1$.}
Under this identification, the exterior derivative $d: \Omega^n_U \to \omega_U$ is identified with a map (of degree $0$)
\eq{}{D:= \alpha\circ d\circ \theta: T_U \to \sO_U.
}

Now suppose given $0 \neq f \in \Gamma(X, L^N)= \Gamma(X, \omega_X^{-1})$. We have a contraction operator
\eq{}{i_{df} : T_U \to \sO_U[N],
}
\comment{Under the Section-Function Isomorphism, we have the identification $\Gamma(U,\cO_U)=\oplus_{k\geq0}\Gamma(X,L^k)$. So, $f\in\Gamma(X,L^N)$ can be regarded as a function on $U$ of degree $N$, and $df$ a 1-form on $U$.}
and we may consider the composition (ignoring the grading)
\eq{}{(\frak g\oplus \C\cdot E)\otimes_\C \sO_U \to T_U \xrightarrow{D+i_{df}} \sO_U.
}
To connect with Theorem \ref{Lie-algebra-homology}, we will now assume that $L$ is {\it very ample} and $G$ is {\it semi-simple}. We claim 
\begin{lem} \label{ghat-action}
 the resulting action
\eq{action-on-CXhat}{(\frak g\oplus \C\cdot E)\otimes_\C \Gamma(U,\sO_U) \to \Gamma(U,\sO_U)
}
is given by
\eq{formula}{x\otimes\phi\mapsto \rho(x)\phi+\phi\rho(x)f,\hskip1in~~~~~~E\otimes\phi\mapsto E(\phi)+\phi E(f)+N\phi}
where $x\in\fg$.
\end{lem}
\begin{proof}Let $\phi, \delta$ be local sections
of $\sO_U, T_U$ respectively. We have
\ga{11b}{\alpha(d\phi\wedge\theta(\delta)) = \delta(\phi); \\
D(\phi\delta) = \phi D(\delta)+\delta(\phi)\label{12b}
}

Let $\rho: \frak g \inj \Gamma(U,T_U)$ be the map given by the action. Since the actions of $G$ and $\G_m$ on $U$ 
commute, the image of $\rho$ lies in the degree $0$ part of $\Gamma(U,T_U)$. Since $D$ has degree $0$, we get
$D\circ\rho(\frak g) \subset \Gamma(U,\sO_U)_{\deg 0} = \C$. Since $D$ commutes with the action of $G$, we get that 
$D\circ\rho[\frak g,\frak g] = (0)$. But $G$ semi-simple implies $\frak g=[\frak g,\frak g]$. 
Thus, $D\circ \rho = 0$. 
\comment{Now it follows from  \eqref{12b} that $D(\phi\delta)=\delta(\phi)$ for $\delta=\rho(*)$. Consider
$$
\fg\otimes\Gamma(U,\cO_U)\ra\Gamma(U,\cO_U)
$$
given by the composition of $\fg\otimes\cO_U\ra T_U\ra\cO_U$ where the first map is the $x\otimes\phi\mapsto\phi\rho(x)$, the second map is $D+i_{df}$. Working this out, we have
$$
x\otimes\phi\in\fg\otimes\Gamma(U,\cO_U)\mapsto\phi\rho(x)\mapsto D(\phi\rho(x))+i_{df}(\phi\rho(x))=\rho(x)\phi+\phi\rho(x)f.
$$
Identifying  $\C[\hat X]=\C[V]/I(\hat X,V)$ with the subring of $\Gamma(U,\cO_U)$ of elements of degree $0~mod~N$, means identifying $\rho(x)$ with $Z^*$, the last expression becomes
$$
Z^*(x)\phi+\phi Z^*(x)f,~~~x\in\fg.
$$
Note that $\beta(\fg)=0$  since $\fg$ is assumed semisimple.}

Taking $\delta=\rho(x)$ in \eqref{12b}, where $x\in \frak g$, we conclude that the diagram
\eq{}{\begin{CD}\Gamma(U,T_U) @> D >> \Gamma(U, \sO_U) \\
@AAA @| \\
\frak g\otimes \Gamma(U,\sO_U) @>\text{natural action}>> \Gamma(U,\sO_U)      
      \end{CD}
}
commutes. This yields the first half of \eqref{formula}.

Next we calculate $D(E)$. Let $S=\Spec A$ be a non-empty open in $X$ such that $L|S \cong \sO_S$. Let $U_S=\pi^{-1}(S)=\Spec A[t,t^{-1}]$. 
(Here $t\in \sO_U$ has degree $1$.) Then $\omega_S = \sO_S\cdot\eta$ for some $n$-form $\eta$, and 
$\alpha^{-1}(1)|U_S = t^Ndt/t\wedge\pi^*\eta$. Restricted to $S$ we have $E=td/dt$, and it is straightforward to check 
from \eqref{11b} that $\theta(E) = t^N\pi^*\eta$. Thus
\ga{}{D(E) = \alpha d \theta(E) = \alpha(Nt^Ndt/t\wedge\eta) = N \\
D(\phi E) = N\phi + E(\phi). 
}

Finally, we have $\alpha(d+df)\theta = D+\alpha(df \theta)$. For $\delta\in T_U$, it follows from \eqref{11b} that 
$\alpha (df)\theta(\delta) = \delta(f)$. 
This yields the second half of \eqref{formula}.
 \comment{
 Let's assume that $H^0(X,\omega_X^{-1})^{\otimes r}\ra H^0(X,\omega_X^{-r})$ is surjective, which is true for $X=G/P$ by Borel-Weil because the target space is irreducible in this case. Then $\omega^{-1}=L^N$ implies that
 $$\C[\hat X]\cong\Gamma(U,\cO_U)_{0 \!\!\!\mod N}.$$
Consider $t\in\cO_U$ which has degree 1, i.e. $E(t)=t$. Then
 $$V^*=H^0(\omega_X^{-1})=H^0(L^N)\subset\C[\hat X]\cong\Gamma(U,\cO_U)_{0 \!\!\!\mod N}$$ 
 has degree $N$.  So, $E=N\sum_i a_i^*{\partial\over\partial a_i^*}$. Note that $t=\ell\in L^*$ locally. In particular,
 $$
 E=-N Z^*(1).
 $$
 Note that $Z(1)$ is the Euler on $V$ where $V$ has degree 1. So $Z^*(1)$ is negative the Euler on $V^*$.
 Let's double check the isomorphism $\C[\hat X]\cong\Gamma(U,\cO_U)_{0 \!\!\!\mod N}$ is compatible with the action of $1\in\hat \fg$. The action \eqref{action-on-CXhat} on $\Gamma(U,\cO_U)$ gives 
 $$E\otimes\phi\mapsto D(\phi E)+i_{df}(\phi E)=N\phi+E(\phi)+\phi E(f)=N\phi+Nr\phi+N\phi f$$
 where $\phi\in\Gamma(X,\cO_X(Nr))\cong \C[\hat X]_r$. The action in Theorem \ref{Lie-algebra-homology} gives
 $$
 Z^*_{f,\beta}(1)\phi=Z^*(1)\phi+\phi Z^*(1)f-\phi=-r\phi-\phi f-\phi.
 $$
 This shows that $Z^*_{f,\beta}(1)$ coincides with $-E/N$ under the action \eqref{action-on-CXhat}. This is exactly where we use the fact that $\beta=(0;1)$.}
\end{proof}

Assume now that $\omega_X = \sO_X(-N)$ for some $N\ge 1$. We view the various graded rings and modules as being graded $\!\!\mod N$,
and the subscript $0 \!\!\!\mod N$ will refer to the sub-object of graded degree zero $\mod N$. For example
\eq{}{\Gamma(U, \sO_U)_{0 \!\!\!\mod N} \cong \bigoplus_{r\ge 0} \Gamma(X, \omega_X^{-r}).
}
\begin{cor}\label{cor3b} Assume \newline\noindent(i) $\omega_X \cong L^{ -N}$ for some $N\ge 1$.
\newline\noindent(ii) The map $\frak g\otimes \sO_X\to T_X$ is surjective. 
 \newline\noindent(iii) The maps 
$\frak g\otimes\Gamma(X,\omega_X^{-r}) \to \Gamma(X,T\otimes\omega_X^{-r})$ are surjective for all $r\ge 0$.

Then the Lie algebra homology $H^{Lie}_0(\hat\fg,R_f)$ in Theorem \ref{Lie-algebra-homology} is isomorphic to
\eq{16b}{\text{Coker}(\Gamma(U, \Omega^n_U)_{0 \!\!\!\mod N} \xrightarrow{d+df}\Gamma(U,\omega_U)_{0 \!\!\!\mod N})
}
\end{cor}
\begin{proof}[proof of corollary] It follows from (ii) together with \eqref{1b} and \eqref{5b} that 
$(\frak g\oplus \C)\otimes \sO_U \surj T_U\cong \Omega^n_U[N]$. Moreover, the coordinate ring of $X\subset\P \Gamma(X,\omega_X^{-1})^*$ is 
$$
\C[\hat X] = \bigoplus_{r\ge 0} \Gamma(X, \omega_X^{-r})\cong \Gamma(U, \sO_U)_{0 \!\!\!\mod N}.
$$
Comparing the $\hat\fg$-actions on both sides (Theorem \ref{Lie-algebra-homology} and Lemma \ref{ghat-action}), we see that this is an isomorphism of $\hat\fg$-module.

Now consider the diagram of global sections viewed as
graded $\mod N$. 
\eq{17b}{\begin{CD}\Gamma(U, \Omega^n_U)_{0 \!\!\!\mod N} @> d+df >> \Gamma(U, \omega_U)_{0 \!\!\!\mod N} \\
@A \theta A \cong A @V \alpha V \cong V \\
\Gamma(U,T_U)_{0 \!\!\!\mod N} @> D+i_{df} >> \Gamma(U,\sO_U)_{0 \!\!\!\mod N} \\
@A \text{surj} AA @| \\
(\frak g\oplus\C)\otimes \Gamma(U,\sO_U)_{0 \!\!\!\mod N} @>>> \Gamma(U,\sO_U)_{0 \!\!\!\mod N}.
\end{CD}
}
It follows from (iii) that the left vertical arrow in \eqref{17b} is surjective, 
so the three horizontal arrows
have isomorphic cokernels. 
\end{proof}

\section{Homogeneous spaces having the surjectivity property}

Concerning condition (iii) of the Corollary \ref{cor3b}. Here is what we can say at the moment.
\begin{prop}\label{prop4b} Let $G$ be a semi-simple group over a field of characteristic zero, and let $P\subset G$ be a parabolic subgroup. Write
$X=G/P$. Let $\sO_X(1)$ on $X$ be very ample with $G$ action and assume $\omega_X = \sO_X(-N)$ for some $N>0$. Write 
$S = \bigoplus_{r\ge 0}\Gamma(X, \sO_X(rN))$, and let $M:= \bigoplus_{r\ge 0} \Gamma(X,T_X(rN))$, so $M$ is a graded $S$-module. 
Then $M$ is generated in degree $0$ in the cases \newline\noindent (a) $\frak g \cong \Gamma(X,T_X)$, and the unipotent 
$\frak u \subset \frak p:=Lie(P)$
is abelian. (In \cite{Ottaviani-Rubei} such $X$ are referred to as Hermitian symmetric. Examples include Grassmannians, quadrics, spinor varieties, maximal Lagrangian Grassmannians, and two
exceptional $X$'s, as well as products of such.) 
\newline\noindent (b) $P=B$ is a Borel. 
\end{prop}
\begin{proof}In case (a), let $\frak l\subset \frak p$ be the Lie algebra of the Levi. The representation of 
$\frak p$ on $\frak g/\frak p$ factors through $\frak p \surj \frak l$. In particular, this representation is completely
reducible.It follows that the tangent bundle $T_X= G\stackrel{P}{\times} \frak g/\frak p$ breaks up as a direct sum of bundles 
$(T_X)_i$ associated to irreducible representations. The same will be true if we tensor with any abelian character of 
$\frak p$. By 
theorem IV, p. 206  \cite{bott},  $\Gamma(X, (T_X)_i)$ is an irreducible
$G$-module. By assumption
\eq{}{\frak g \cong \Gamma(X,T_X) \cong \bigoplus_i\Gamma(X,(T_X)_i)
}  
so we get a decomposition $\frak g = \bigoplus \frak g_i$ of the adjoint representation of $G$ on $\frak g$. Again by Bott,
$\Gamma(X,(T_X)_i\otimes\omega_X^{-r})$ is $G$-irreducible for $r>0$. Finally, the maps 
$\frak g_i \otimes\Gamma(X, \omega_X^{-r}) \to \Gamma(X,(T_X)_i\otimes\omega_X^{-r})$ are non-zero and hence surjective since
they are compatible with the $G$-action.

Suppose now $G$ is arbitrary semi-simple and $P=B$ is a Borel with Lie algebra $\frak b$, maximal torus $T$ with Lie algebra 
$\frak t$, and unipotent radical $U$ with Lie algebra $\frak u$. 
\begin{lem} Let $M$ be a $G$-module. Then the vector bundle $G\stackrel{P}{\times} M \cong \sO_X\otimes_\C M$. 
\end{lem}
\begin{proof}[proof of lemma]The bundle is obtained from $G\times M$ by identifying $(gb,m)\sim (g,bm)$. A section of $G\times M$ 
is given by a function $f: G \to M$. Since $(gb,f(gb))\sim (g,bf(gb))$, this section descends to a section of 
$G\stackrel{P}{\times} M$ if and only if $f(gb) = b^{-1}f(g)$. If the representation of $P$ on $M$ lifts to $G$, we define
for $m\in M$, $f_m(g) := g^{-1}m$. Then $f_m(gb) = b^{-1}f_m(g)$, so $f_m$ descends to a section of the bundle. In this way, 
we obtain a trivialization. 
\end{proof}

Applying the lemma to the exact sequence of $\frak b$-modules $0 \to \frak b \to \frak g \to \frak g/\frak b \to 0$ yields an
exact sequence of bundles
\eq{19b}{0 \to \sV \to \frak g\otimes_\C \sO_X \to T_X \to 0. 
}
The exact sequence $0 \to \frak u \to \frak b \to \frak t \to 0$ yields another exact sequence of bundles
\eq{20b}{0 \to \Omega^1_X \to \sV \to \frak t\otimes_\C \sO_X \to 0
}
(Note that the vector bundle associated to $\frak u$ is $\Omega^1_X$. Also $\frak t$ has trivial $\frak b$-action and hence by
the lemma the corresponding equivariant bundle is trivial.) 

For the proof of proposition \ref{prop4b} we need 
\eq{22b}{H^1(X, \sV(rN))=(0),\ r\ge 0. 
}
When $r=0$ this is true since $\frak g \cong \Gamma(X, T_X)$. Assume $r\ge 1$. The character $c$ of $\frak b$ associated to 
$\omega_X^{-1} = \sO_X(N)$ is the sum of all the positive roots of $\frak g$. (We follow the notation \cite{Humphreys}. $(\alpha,\beta)$ denotes the Killing form, and 
$\langle\alpha,\beta\rangle := 2(\alpha,\beta)/(\beta,\beta)$.) One knows (\cite{Humphreys} p. 50) that $\langle c,\alpha\rangle = 2$
for any simple positive root $\alpha$. For $\alpha, \beta$ both positive, we have $\langle \beta,\alpha\rangle \le 3$ 
(\cite{Humphreys} p. 45),so $\langle c-\beta,\alpha\rangle \ge 2-3=-1$. Let $\sL(c-\beta)$ be the line bundle on $G/B$ associated
to the character $c-\beta$. We have the following consequences of Borel-Bott-Weil theory (\cite{Demazure},
corollaire 8).\newline\noindent(i) If $\langle c-\beta,\alpha\rangle \ge 0$ for all positive simple $\alpha$, then 
$H^*(X, \sL(c-\beta)) = (0),\ *\ge 1$. \newline\noindent(ii) If there exists a simple positive $\alpha$ for which 
$\langle c-\beta,\alpha\rangle = -1$, then $H^*(X, \sL(c-\beta)) = (0)$ for all $*\ge 0$. \newline\noindent(iii) For $n\ge 2$,
the character $nc-\beta$ is dominant, so $H^*(X, \sL(nc-\beta))=(0), *\ge 1$. 

Start with the identity $\Omega^1_X = G \stackrel{B}{\times} \frak u$, the vector bundle $\Omega^1_X(rN)$ has a filtration with
 quotients which
are line bundles of the form $\sL(rc-\beta)$ as above. Since all these have vanishing cohomology in degrees $\ge 1$, it follows
that the same will be true for $\Omega^1_X(rN)$. Since the line bundle $\sO_X(rN)$ corresponds to a dominant weight for $r\ge 1$,
the desired vanishing \eqref{22b} now follows from \eqref{20b}. 
\end{proof}

\section{From Lie algebra homology to de Rham cohomology}

In this section, we will prove Theorem \ref{geo.thm}.

With notation as above, define
\eq{}{W = X-\sV(f).
}

It remains to interpret the cokernel in \eqref{16b} in terms of the cohomology of $W$ in middle degree $n$. For this we adapt a 
method of Dimca \cite{Dimca}. $X \subset \P^n$ will be a smooth, projective variety with cone 
$\Spec R$, so $R=\C[x_1,\dotsc,x_{n+1}]/I$ for a homogeneous ideal $I$. Let $B=\C[x_1,\dotsc,x_{n+1}]$. 
$\Delta = \sum x_i \frac{\partial}{\partial x_i}$ will be the Euler operator which we view as acting by contraction 
$\Delta: \Omega^i_B \to \Omega^{i-1}_B$. The following properties of $\Delta$ are elementary:
\begin{lem}\label{lem3}(i) $\Delta^2=0$. The complex 
$\Omega^{n+1}_B \xrightarrow{\Delta}\Omega^n_B \to \cdots \to \Omega^1_B \to B$ can be identified with the Koszul complex 
associated to the ideal $(x_1,\dotsc,x_{n+1}) \subset B$. It is acyclic away from $0\in \Spec B$. 
\newline\noindent (ii)The $B$-module $\Omega^i_B= \bigoplus_{r\ge i} \Omega^i_{B,r}$ is graded, with $x_i$ and $dx_i$ of 
degree $1$. 
$\Delta$ and the exterior differential $d$ have degree $0$ for this grading, and $d\Delta+\Delta d = \mu$ is the number 
operator for this grading, acting by multiplication by $r$ on $\Omega^i_{B,r}$. In particular, if $f\in B$ is homogeneous 
of degree $N$, then $\Delta df = Nf$. 
 \newline\noindent (iii) For $u\in \Omega^i_B, v\in \Omega^j_B$ we have 
$\Delta(u\wedge v) = \Delta(u)\wedge v + (-1)^iu\wedge \Delta(v)$.
\newline\noindent (iv) Let $I \subset B$ be a homogeneous ideal, and let $\sI^* \subset \Omega^*_B$ be the differential ideal 
generated by $I$. Then $\Delta(\sI^*) \subset \sI^*[-1]$. In particular $\Delta$ induces a map of graded sheaves 
$\Delta_R: \Omega^*_R \to \Omega^*_R[-1]$, where $R:= B/I$.
\end{lem}
\begin{proof}(ii) and (iii) are proved in detail in Dolgachev, Weighted projective varieties, lemma in section 2.1.3. 
(i) is immediate, and (iv) is clear from the last assertion in (ii). 
\end{proof}

Let $U=\Spec R -\{0\}$ be the punctured cone. We have a $\G_m$-bundle $\pi: U \to X$. 
As a consequence of lemma \ref{lem3}, $\Delta$ induces a surjection $\Delta_U: \Omega^1_U \surj \sO_U$, and the induced complex
\eq{14}{\Omega^{r+1}_U \xrightarrow{\Delta_U} \cdots \to \sO_U 
}
is the corresponding Koszul complex and is acyclic. Note $\Delta_U(\pi^*\Omega^1_X) = (0)$ (it suffices to remark for $z$ homogeneous of degree $0$ that $\Delta_U(dz)=0\cdot z = 0$). 
It follows by looking at ranks that the sequence 
\eq{15}{0 \to \pi^*\Omega^1_X \to \Omega^1_U \xrightarrow{\Delta_U} \sO_U \to 0
}
is exact, and from \eqref{14} and \eqref{15} that 
\eq{16}{\pi^*\Omega^i_X = \text{Image}(\Delta_U: \Omega^{i+1}_U \to \Omega^i_U). 
}
(Alternatively, we can identify $\Delta: \Omega^1_U \surj \sO_U$ in \eqref{15} with 
$\Omega^1_U \surj \Omega^1_{U/X}\cong \sO_U\cdot dt/t$. the complex \eqref{14} then becomes the Koszul complex on this arrow.)

We get exact sequences of sheaves on $U$ and on $X$
\ga{}{0 \to \pi^*\Omega^i_X \to \Omega^i_U \to \pi^*\Omega^{i-1}_X \to 0 \\
0 \to \bigoplus_\Z \Omega^i_X(n) \to \pi_*\Omega^i_U \to \bigoplus_\Z \Omega^{i-1}_X(n) \to 0.\label{18}
}
Note that $\pi_*\Omega^i_U$ is $\Z$-graded (locally $U\cong X\times \Spec \C[t,t^{-1}]$ 
and we give $t$ degree $1$ and $dt/t$ degree $0$. The resulting grading is 
independent of the choice of $t$. (Better said, there is a $\G_m$-action on $U/X$.)
The exact sequence \eqref{18} is compatible with the grading. For convenience we will assume $\Gamma(X, \Omega^{>0}_X(n)) = (0)$ for $n\le 0$. 
It follows that for $i\ge 1$, $\Gamma(U, \Omega^i_U)$ is graded in degrees $> 0$. For $\omega \in \Gamma(U, \Omega^i_U)$ homogeneous, 
we write $|\omega|$ for the homogeneous degree. 

Let $f\in R_N$ be a non-zero homogeneous function of degree $N\ge 1$. For integers $a, s, t$ with $0\le a\le N-1$, we define
\eq{}{B^{s,t}_a := \Gamma(U, \Omega^{s+t+1}_U)_{Nt-a}.
}
(The subscript on the right refers to the homogeneous degree of the form.) Let $d':B^{s,t}_a \to B^{s+1,t}_a$ be the exterior derivative, 
and define $d_a{}''(\omega) =-tdf\wedge\omega$ for $\omega \in B^{s,t}_a$. This defines $d_a{}''(\omega)$ for $\omega$ homogeneous, and we extend the definition to all forms by linearity. We have $d_a{}'': B^{s,t}_a \to B^{s,t+1}_a$. Note $d'd_a{}''=-d_a{}''d'$, so the graded vectorspace $B^*_a = \bigoplus_{s+t=*}B^{s,t}_a$ 
is a complex with differential $\delta_a:=d'+d_a{}'': B^*_a \to B^{*+1}_a$. We write
\ga{}{D_f = \delta_0: B^*_0 \to B^{*+1}_0; \\
\sigma: B^{s,t}_a \to \Gamma(W,\Omega^{s+t}_W(-a));\quad \sigma(\omega) = \frac{\Delta\omega}{f^t}.
}
\begin{lem} Let $d_W:\Gamma(W,\Omega^{s+t}_W(-a))\to \Gamma(W,\Omega^{s+t+1}_W(-a))$ be exterior differentiation. Then $(d_W\circ\sigma + \sigma\circ \delta_a)(\omega) = \frac{-a\omega}{f^t}$. In particular, when $a=0$,  $\sigma$ induces
a map on cohomology $\sigma: H^*(B^*_0,D_f)  \to H^*_{dR}(W)$.  
 \end{lem}
\begin{proof}
 We have for $\omega\in B^{s,t}_a$
\ga{}{d_W\sigma(\omega) = d_W(\Delta(\omega)/f^t))= (fd_W\Delta(\omega)-tdf\wedge\Delta\omega)/f^{t+1} \\
\sigma\delta_a(\omega) = \sigma(d\omega-tdf\wedge\omega) = \frac{f\Delta d\omega - tNf\omega + tdf\wedge \Delta\omega}{f^{t+1}} \notag
}
Combining these, and using lemma \ref{lem3}(ii)
\eq{}{(d_W\sigma+ \sigma\delta_a)(\omega) = \frac{-a\omega}{f^t}.
}
\end{proof}

\begin{lem}Assume that $B^{i,0}_0 = (0)$ for $i\ge 0$ (i.e. $\Gamma(X,\Omega^j)=(0)$ for $j\ge 1$.) Then $B^*_0$ with differential $D_f = d-tdf$ is quasi-isomorphic to $B^*_0$ with differential $d-df$. 
\end{lem}
\begin{proof} Define constants $\mu_{s,t}$ recursively by $\mu_{s,1} = 1$ and $t\mu_{s,t+1}=\mu_{s,t}$. (More simply, $\mu_{s,t} = 1/(t-1)!$.) The diagrams 
\eq{}{\begin{CD}B^{s,t}_0 @> d-tdf>> B^{s+1,t}_0\oplus B^{s,t+1}_0 \\
@V \mu_{s,t} VV   @V \mu_{s+1,t}\oplus \mu_{s,t+1} VV \\
B^{s,t}_0 @> d-df>> B^{s+1,t}_0\oplus B^{s,t+1}_0 
\end{CD}
}
all commute. Note our assumption means we need only consider the case $t\ge 1$. 
\end{proof}

\begin{thm}
We continue to assume $\Gamma(X, \Omega^i)=(0)$ for $i\ge 1$. Then the map $\sigma$ induces an isomorphism on cohomology 
$H^*(B^*_0,D_f) \to H^*_{dR}(W)$. 
\end{thm}
\begin{proof}
 We have a decreasing filtration 
\eq{}{F^pB^k_0 := \bigoplus_{s\ge p}B^{s,k-s}_0;\quad D_f(F^pB^k_0)\subset F^pB^{k+1}_0.
}
We define a filtration on $\Gamma(W, \Omega^*_W)$ by
\eq{}{F^p\Gamma(W, \Omega^j_W) = \begin{cases}\{\omega/f^{j-p}\ |\ \text{$\omega$ has no pole along $f=0$}\} & j\ge p \\
                                  0 & p>j
                                 \end{cases}
}
Again $\quad dF^p \subset F^p$. 
We have $\sigma(F^pB^k_0)\subset F^p\Gamma(W, \Omega^k_W)$ and a map of spectral sequences
\eq{}{\sigma: {}'E^{p,q}_1 := H^{p+q}(gr_F^pB^*_0) \to E^{p,q}_1:= H^{p+q}(gr^p_F\Gamma(W,\Omega^*_W))
}
Let $\eta/f^{j-p}\in F^p\Gamma(W,\Omega^j_W)$ represent a class in $H^{j}(gr^p_F\Gamma(W,\Omega^*_W))$. (By our hypothesis, $j>p$.)
Then
\eq{}{d(\eta/f^{j-p}) = (fd\eta-(j-p)df\wedge\eta)/f^{j-p+1}
}
is divisible by $f$, i.e. $\theta:= -\frac{|\eta|}{N}\frac{df}{f}\wedge\eta \in B^{j}_0$. $H^{j}(gr_F^pB^*_0)$ is the 
cohomology of the complex
\eq{}{B^{p,j-1-p}_0 \xrightarrow{df\wedge} B^{p,j-p}_0 \xrightarrow{df\wedge} B^{p,j+1-p}_0 
}
The form $\eta$ is well-defined upto a multiple of $f$ and a form $df\wedge x$, so $\theta$ is well-defined upto a form 
$df\wedge\xi$. Since $df\wedge\theta=0$, 
we see that $\theta$ represents a well-defined class in $H^{j}(gr_F^pB^*_0)$, and $\sigma(\theta) = |\eta|\eta/f^{j-p}$. It follows that
$\sigma$ induces an isomorphism on $E_1$ terms, and hence on $E_r$-terms for any finite $r$. To conclude we remark that 
$B^*_0 = \varinjlim_{p\to -\infty}F^pB^*_0$ and $\Gamma(W,\Omega^*_W) = \varinjlim_{p\to -\infty}F^p\Gamma(W,\Omega^*_W)$. For any finite value
of $p$, there are induced spectral sequences on $F^p$. For any one of these we have $E_1^{u,v}, {}'E_1^{u,v}$ vanishing for $u<<0$. Again $\sigma$ 
will induce isomorphisms on $E_1$. It follows that $\sigma$ is a direct limit of isomorphisms and hence is an isomorphism. 
\end{proof}
\begin{cor}\label{deRham}
We have (notation as in corollary \ref{cor3b}) 
\eq{}{ H^n_{dR}(W) \cong H^{Lie}_0(\frak g\oplus \C\cdot E, \Gamma(U,\sO_U)_{0\!\!\!\mod N})\cong H^{Lie}_0(\hat\fg,R_f).
}
\end{cor}

\noindent{\it Proof of Theorem \ref{geo.thm}.} This follows immediately from Corollary \ref{deRham} and Theorem \ref{Lie-algebra-homology}. $\Box$

\section{Solution sheaf vs. period sheaf}

As an application, we will compare the solution sheaf of our tautological system and the period sheaf of smooth CY hyperplane sections in $X$, by giving a completeness criterion for the tautological system. We will then deduce Conjecture \ref{HLY-conjecture} as a special case.

\begin{defn}
We say that $\cM=\tau(\hat X,\Gamma(X,\omega_X^{-1})^*,G\times\G_m,(0;1))$ is complete, if its solutions sheaf coincides with the period sheaf $\Pi(X)$ on $\Gamma(X,\omega_X^{-1})_{sm}$.
\end{defn}

\begin{cor} 
The tautological system 
$$\cM=\tau(\hat X,\Gamma(X,\omega_X^{-1})^*,G\times\G_m,(0;1))$$ 
is complete iff the primitive cohomology $H^n(X)_{prim}$ is zero.
\end{cor}
\begin{proof}
We have the exact sequence
\comment{See p159 [Voisin2].}
$$
0\ra H^n(X)_{prim}{\br j^*\over\ra} H^n_{dR}(W){\br\Res\over\ra} H^{n-1}(Y_f)_{van}\ra0
$$
where $H^{n-1}(Y_f)_{van}$ is the vanishing cohomology of $Y_f=\cV(f)$. Thus
$\Res$ is an isomorphism iff $H^n(X)_{prim}=0$. By Proposition \ref{period-rank}, the dimension of the vanishing cohomology coincides with the rank of the period sheaf. By Theorem \ref{Lie-algebra-homology} and Corollary \ref{deRham}, this agrees with the generic rank of our system $\cM$ iff $H^n(X)_{prim}=0$.
\end{proof}

\begin{prop} \label{period-rank}
The rank of $\Pi(X)$ is equal to the dimension of the vanishing cohomology of a smooth CY hypersurface $Y_f$.
\end{prop}
\begin{proof}
 Fix a smooth CY hyperplane section $f$. We know that the monodromy representation on the vanishing cohomology $H^{n-1}(Y_f)_{van}$ is irreducible. It follows that the monodromy action on $H_{n-1}(Y_f)/H^{n-1}(Y_f)_{van}^\perp$ is also irreducible. We have a nonzero homomorphism of representations from $H_n(Y_f)$ to the stalk of $\Pi$ at $f$:
$$
H_{n-1}(Y_f)\ra \Pi_f,~~~\gamma\mapsto\int_\gamma \Res \Omega_f
$$
where $\Res \Omega_f$ is the Poincar\'e residue of a meromorphic form on $X$ with pole along $Y_f$ (Theorem 6.6 \cite{LY}.) Since $\Res \Omega_f\in H^{n-1}(Y_f)_{van}$, it follows that $H^{n-1}(Y_f)_{van}^\perp\subset H_{n-1}(Y_f)$ lies in the kernel of map. By irreducibility, the map induces
$$
H_{n-1}(Y_f)/H^{n-1}(Y_f)_{van}^\perp\cong\Pi_f.
$$
\comment{Note that since $H_{n-1}(Y_f)$ is rigid (topological), we can identify canonically all such groups for $\sigma$ close to $f$. So $H^{n-1}(Y_f)_{van}^\perp\subset H_{n-1}(Y_f)\equiv H_{n-1}(Y_f)$ is independent of $\sigma$.}
\end{proof}

\begin{cor}
Conjecture \ref{HLY-conjecture} holds. Therefore, the generic rank of the solution sheaf of the tautological system in this case is 
$${n\over n+1}(n^n-(-1)^n).$$
\end{cor}
\begin{proof}
For $X=\P^n$, Corollary \ref{deRham} holds in this case (see Proposition \ref{prop4b}), and  we obviously have $H^n(X)_{prim}=0$. So the tautological system $\cM$ in the preceding corollary is complete, proving Conjecture \ref{HLY-conjecture}. The last assertion is an easy calculation of $\dim H^{n-1}(Y_f)$ using the Lefschetz hyperplane theorem.
\end{proof}

\section{A chain map}

Corollary \ref{deRham} suggests that there might be a similar relation between Lie algebra homology groups and de Rham cohomology groups in degrees other than 0 and $n$. In this section, we define a chain map between the complexes defining those (co)homology groups. 

Recall that the Lie algebra homology of $\hat\fg$ with coefficient in $R_f$ can be given as the homology of the chain complex $(C_*(\hat\fg, R_f),d_{CE})$ where $d_{CE}$ is the Chevelley-Eilenberg homology differential and
$$
C_p(\hat \fg,R_f):=(U\hat \fg\otimes_\mathbb{C} \wedge^p \hat \fg)\otimes_{U\hat\fg}R_f\cong \wedge^p\hat\fg\otimes_\mathbb{C}R_f.
$$
We will define a chain map 
\eq{chain-map}{\varphi:(C_*(\hat\fg, R_f),d_{CE})\rightarrow (\Gamma(U, \Omega^{n+1-*}_U)_{0 \!\!\!\mod N}, d+df\wedge-)}
(where $N$ will be 1) which induces the isomorphism in (co)homology in one degree given by Corollary \ref{cor3b}. Here the second complex 
extends (horizontally) the top row of \eqref{17b}.
\question{Is it true that the cohomology of the second complex coincides with $H_{dR}^*(X-Y_f)$? Spencer's notes show that this de Rham cohomology coincides with $H^*(B_0^*,D_f)$, but $B_0^*$ is a bit different from the second complex above.}

As in section \ref{LA-geometry}, we choose $L=\omega_X^{-1}$, $U=L-\{0\}$, so that we have the identification $R_f=\Gamma(U,\cO_U)$. The space $U$ has a unique (up to scalar) $G$-invariant nonvanishing holomorphic top form $\omega_1$ such that 
$$h\cdot\omega_1=h^{-1}\omega_1$$
for $h\in\G_m$ (Theorem 3.3 \cite{LY}.) Then a straightforward calculation yields

\begin{prop}
Define \eqref{chain-map} by
$$
\varphi(x_1\wedge\cdots\wedge x_p\otimes g)=g i_{x_1}\cdots i_{x_p}\omega_1
$$
where $g\in R_f$, $x_j\in\hat\fg$, and $i_{x_j}$ denotes the contraction with the vector field generated by $x_j$. Then $\varphi$ is a chain map. Moreover, $\varphi$ is surjective iff the contraction map
$$
\hat\fg\otimes\Gamma(U, \Omega^{n+1-p}_U)\ra\Gamma(U, \Omega^{n-p}_U),~~~x\otimes\lambda\mapsto i_x\lambda
$$
is surjective for each $p\geq0$.
\end{prop}

Note that the $p=0$ surjectivity condition above is equivalent to condition (iii) of Corollary \ref{cor3b}. We expect that $\varphi$ is surjective in general. However it need not induce an isomorphism on all (co)homology groups. In any case, the subcomplex $\ker(\varphi)\subset C_*(\hat\fg,R_f)$ can be described as follows. 

First, note that the $C_p(\hat\fg,R_f)=\wedge^p \hat \fg\otimes_\mathbb{C}\C[\hat X]$ as vector spaces and $\varphi$ as a linear map are both independent of $f$. The dependence on $f$ is through the differential of the complex. Put $S_0:=0$, and for $p\geq1$, define $S_p\subset C_p(\hat\fg,R_f)$ inductively by
\begin{equation}
S_p:=\cap_{g\in \Gamma(X,\omega_X^{-1})}d_g^{-1}(S_{p-1})
\end{equation}
where $d_g:C_p(\hat\fg,R_g)\ra C_{p-1}(\hat\fg,R_g)$ denotes the Chevellay-Eilenberg differential for a given $g\in\Gamma(X,\omega^{-1})$. In other words, given $c\in C_p(\hat\fg,R_f)$, we have $c\in S_p$ iff $d_g c\in S_{p-1}$ for all $g\in \Gamma(X,\omega_X^{-1})$. Clearly $S_*\subset C_*(\hat\fg,R_f)$ is a subcomplex. (Again, as a subspace it is clearly independent of $f$.)


We claim that $S_*=\ker(\varphi)$. This follows from the following standard argument:

At degree $p=0$, $\ker(\varphi)=0$ is obvious. To see that at degree $p$, $\ker(\varphi)=S_p$, note that $\ker(\varphi)\subseteq S_p$ is easy. On the other hand, for any $\beta\in S_p$, by induction $\varphi(\beta)$ is in the kernel of $d+dg\wedge.$ for any $g\in H^0(X,-K_X)$, then it is easy to show that $\varphi(\beta)$ has to be zero, and thus $S_p\subseteq ker(\varphi)$. 
\question{For $X=\P^n$, this is OK. But to show $S_p\subset\ker(\varphi)$ in general, we must show that  if $d(\varphi_p c)=0$ and $dg\wedge\varphi_p c=0$ for all $g$ then $\varphi_p c=0$. Is this obvious?}

\section{Rank 1 points for $G(2,N)$}\label{add.rk1}

In this section, we give an example of a rank $1$ point for $G(2,N)$. We give two proofs of their rank 1 property: one combinatorial and another geometric proofs. The first one is long, but we include it because it suggests an interesting connection between coinvariant space of Lie algebra module and graph theory. The second proof uses our geometric formula for holonomic rank and is much simpler.

\begin{thm} \label{rank1-points}
Let $X=G(2,N)$, $\hat G=SL_N\times\G_m$. At the hyperplane section $f=x_{1,2}\cdots x_{N-1,N}x_{N,1}$ (where the $x_{ij}$ are the Pl\"ucker coordinates of $X$), the rank of the solution sheaf to $\cM:=\tau(\hat X,\Gamma(X,\omega_X^{-1})^*,\hat G,(0;1))$ is 1.  
\end{thm}

We will first give the combinatorial proof in a series of lemmas.

\begin{lem}\label{at-most-rank-one}
The rank of the solution sheaf of $\cM$ at $f$ is at least 1.
\end{lem}

\begin{proof} 
Consider the affine coordinate chart on $X$ given by the $2\times N$ matrix of the form $[I|Z]$, where $I$ is the $2\times 2$ identity matrix and $Z=(z_{ij})$. Let $T_r$ be the real torus of dimension $2(N-2)$ defined by
$$
|z_{ij}|=r_{ij}.
$$
For generic $r_{ij}>0$, $T_r$ lies in $X-\sV(f)$. Likewise, it lies in $X-\sV(f')$, for $f'$ close to $f$. Now the period integral $\Pi(f'):=\int_{T_r}{\Omega\over f'}$, as defined in \cite{LSY}, is a local solution to our tautological system in a neighborhood of $f$. We will show that this solution is nonzero. We can write $f'$ in the form
$$
f'=af+\sum_\alpha a_\alpha x^\alpha
$$
where the sum is over monomials (in Pl\"ucker coordinates) $x^\alpha\neq f$ in $\Gamma(X,\omega_X^{-1})$, and view 
the solution as a function of the variables $a_\alpha$ close to $0$, and of $a$ close to $1$. We can expand $\Pi(f')$ as a power series whose leading term is $a^{-1}\Pi(f)$. By using the global residue formula for $X=G(2,N)$ \cite{LSY}, we find that the leading term of $\Pi(f)$, as $a_\alpha\ra0$, is nonzero.
This shows that the solution $\Pi(f')$ is nontrivial, and so the holonomic rank at $f$ is at least 1.
\end{proof}

Put $V=\Gamma(X,\omega_X^{-1})^*$. Then 
$$
R:=\C[V]/I(\hat X,V)\cong\oplus_{k\geq0}\C[W]_{kN}/I(\hat X,W)_{kN}
$$
where $X\subset\P W$, $W=\wedge^2\C^N$, under the Pl\"ucker embedding. Thus we can represent elements of $R$ as 
polynomials of degrees divisible by $N$, in the Pl\"ucker coordinates, subject to the Pl\"ucker relations. By Theorem \ref{Lie-algebra-homology}, the solution sheaf of $\cM$ at $f$ is isomorphic to the dual of the coinvariant space
$$H^{Lie}_0(\hat\fg,Re^f)=Re^f/\hat\fg(Re^f).$$
We will show that the coinvariant space above has dimension at most 1, and is generated by the monomial $f$. We begin with some preparations. Let $T$ be the diagonal maximal torus of $SL_N$.
Since the monomial $f$ is $T$-invariant, for any weight vector $h\in R$ of $T$, $he^f$ is also a weight vector of the same weight. Since a weight vector is an eigenvector of the Lie algebra $\ft$ of $T$, it follows $he^f$ is trivial in the coinvariant space unless the weight is zero. Thus to prove the theorem, it suffices to show that 

(*) {\it $he^f\equiv\text{const.}~ e^f$ $\mod$ $\hat\fg(Re^f)$, for each $T$-invariant $h\in R$.}

We introduce a graphical representation of $R$ as follows. Each monomial in the $x_{ij}$ is represented as a graph: it has $N$ vertices, labelled by elements of $\Z/N\Z$ which we denote by $1,2,..,N$, placed evenly on a circle in counterclockwise order. See figure 1 below. A given monomial $x_{i_1j_1}\cdots x_{i_n,j_n}$ corresponds to the graph with $n$ edges, connecting pairs of vertices which are labelled by $(i_1,j_1),..,(i_n,j_n)$ respectively, repetition allowed. In other words, each factor $x_{ij}$ corresponds to a single edge connecting vertices $i$ and $j$. We identify $\C[W]$ with $\cG$, the linear space generated by all such graphs whose numbers of edges are divisible by $N$. 

For $1\leq i_1<i_2<i_3<i_4\leq N$, we have a  Pl\"ucker relation
$$
x_{i_1 i_3}x_{i_2i_4}=x_{i_1i_2}x_{i_3i_4}+x_{i_1i_4}x_{i_2i_3}
$$
which is graphically given by figure 1. We can view this relation as an {\it operation} to be performed on a graph to remove the crossing of the two edges $(i_1,i_3),(i_2,i_4)$, thereby expressing a given graph with a crossing, as a linear combination of graphs without it. Note that such an operation does not change the valence at each vertex. Since $I(\hat X,W)$ is generated by the Pl\"ucker relations, an element of $\cG$ lies in $\cI:=I(\hat X,W)$ iff it can be reduced to zero by a finite number of such Pl\"ucker operations.

\begin{figure}[htb]
\vskip-1in
\hskip-2in\includegraphics[height=7in,width=5.5in,angle=0]{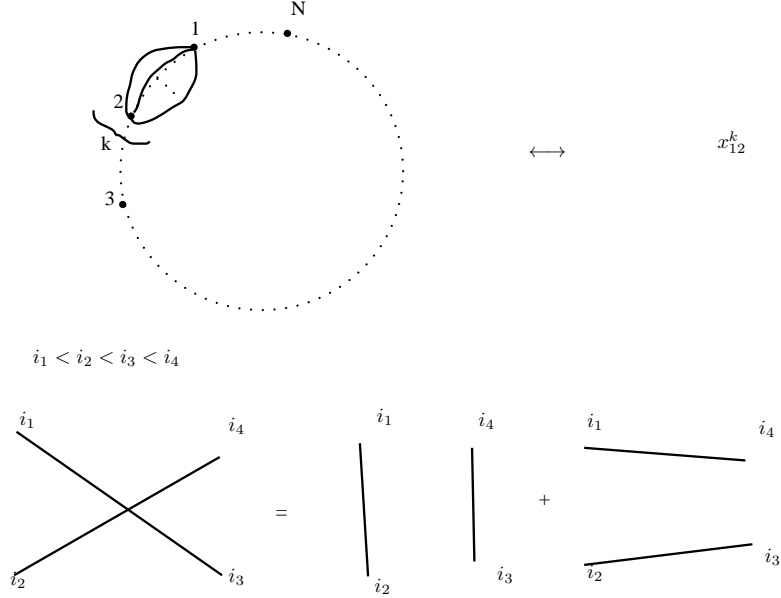}
\vskip-3in\caption{Graph-monomial correspondence and Pl\"ucker relations.}
\end{figure}


Define two integer valued functions $I_a$ and $I_m$ on graphs:
\begin{equation} \label{I_a}
I_a(G)=\sum_{\text{edges } e}d(e)
\end{equation}
\begin{equation}\label{I_m}
I_m(G)=\prod_{\text{edges } e}d(e)
\end{equation}
where for an edge $e$ connecting vertices $i,j$ in $G$, we put $d(e):=\min_{k\in\mathbb{Z}}|i-j+kN|$. Note that if we assign a distance 1 between any two neighboring vertices on the circle, then $d(e)$ is just the shortest distance between vertices $i,j$.

\begin{lem}
Let $G$ be a graph with at least one crossing, and $G_1+G_2\in\cG$ be obtained from $G$ by a single Pl\"ucker operation to remove a crossing. Then one of the following holds:
\begin{itemize} 
\item (1) $I_a(G_1)<I_a(G)$ and $I_a(G_2)<I_a(G)$;
\item (2) $I_a(G_1)<I_a(G)$ and $I_a(G_2)=I_a(G)$ and $I_m(G_2)<I_m(G)$;
\item (3) $I_a(G_1)=I_a(G)$ and $I_a(G_2)<I_a(G)$ and $I_m(G_1)<I_m(G)$.
\end{itemize}
\end{lem}
\begin{proof}
Suppose $G$ contains the subgraph (corresponding to) $x_{i_1 i_3}x_{i_2i_4}$, with a crossing, and we perform a Pl\"ucker operation to remove it. After the Pl\"ucker operation (Firgure 1), we obtain a sum of two graphs $G_1,G_2$, which contain the subgraphs (without crossings) $x_{i_1i_2}x_{i_3i_4}$, $x_{i_1i_4}x_{i_2i_3}$ respectively. We have that $d(i_1,i_3)$ is equal to either $d(i_1,i_2)+d(i_2,i_3)$ or $d(i_3,i_4)+d(i_4,i_1)$ (cf. fig. 1.) Likewise $d(i_2,i_4)$ is either $d(i_2,i_3)+d(i_3,i_4)$ or $d(i_1,i_2)+d(i_4,i_1)$. In each of 4 cases, one finds that
\begin{eqnarray*}
&d(i_1,i_3)+d(i_2,i_4)\geq d(i_1,i_2)+d(i_3,i_4)~~\&\cr 
&d(i_1,i_3)+d(i_2,i_4)\geq d(i_1,i_4)+d(i_2,i_3)
\end{eqnarray*}
and that at least one of the inequalities is strict. 
Without loss of generality, we can assume that $I_a(G_1)<I_a(G)$.
So, it remains to show that if $I_a(G_2)=I_a(G)$ then $I_m(G_2)<I_m(G)$. For this shows that that either (1) or (2) holds ( (1) or (3) if we interchange the roles of $G_1,G_2$.) Thus assume $I_a(G_2)=I_a(G)$, which means that
\begin{equation}\label{case1}
d(i_1,i_3)+d(i_2,i_4)=d(i_4,i_1)+d(i_2,i_3).
\end{equation}
Consider the case $d(i_1,i_3)=d(i_1,i_2)+d(i_2,i_3)$ and $d(i_2,i_4)=d(i_2,i_3)+d(i_3,i_4)$. Then
$$
d(i_1,i_3)d(i_2,i_4)=(d(i_1,i_2)+d(i_2,i_3))(d(i_2,i_3)+d(i_3,i_4)).
$$
The right side is $>(d(i_1,i_2)+d(i_2,i_3)+d(i_3,i_4))d(i_2,i_3)=d(i_4,i_1)d(i_2,i_3)$ by
eq. \eqref{case1}. This shows that $I_m(G)>I_m(G_2)$ in this case. The remaining 3 cases can be  done similarly. 
\end{proof}

\begin{cor}\label{no-crossings}
Every graph $G$ can be reduced to a linear combination of graphs with no crossings. 
\end{cor}
\begin{proof}
Let $NC$ be the span of graphs with no crossings.
By the preceding lemma, if $I_a(G)\leq 1$ then $G\in NC$. Assume that $I_a(G)\leq k$ implies $G\in NC$.
Let $I_a(G)=k+1$. Then $G$ can be reduced to $G_1+G_2$, where $G_1\in NC$ by the lemma again, so
$$
G\equiv G_2\mod NC.
$$
If $I_a(G_2)\leq k$ then we are done. If not, then we have $I_a(G_2)=I_a(G)=k+1$ and $I_m(G_2)<I_m(G)$ by the lemma. Continuing this way, the reduction process must terminates at some point with $G\equiv 0\mod NC$, or
$$
G\equiv G_q\mod NC
$$
where $G_q$ is a graph with $I_a(G_q)=I_a(G)$ and $I_m(G_q)$ can no longer be decreased by a Pl\"ucker operation. In this case $G_q$ has no crossings, hence $G\in NC$. 
\end{proof}

A degree $N$ monomial $x_{i_1j_1}\cdots x_{i_Nj_N}$ is $T$-invariant iff each $i\in\{1,..,N\}$ appears exactly twice as subscripts in this monomial. Thus the graph  of this monomial has valence 2 at each vertex. Similarly, a degree $mN$ monomial is $T$-invariant iff it has valence $2m$ at each vertex. We call it a valence $2m$ graph.
We shall say that a graph has a {\it cyclic} loop of length $s\geq2$ if there is a vertex $i$ and an edge between any two consecutive vertices in the following list:
$$
i,i+1,...,i+s-1,i.
$$
In this case, we say that the loop contains those vertices.
Clearly a cyclic loop of length $N$ corresponds to the monomial $f=x_{1,2}\cdots x_{N-1,N}x_{N,1}$.

Suppose $G$ is a valence $2m$ graph with no crossings. Then it is easy to see that every vertex in $G$ belongs in one or more loops, not necessarily cyclic. 
Observe that a non-cyclic loop in $G$ must contain an edge $(a,b)$ such that there is at least one loop on each side of it. By minimizing the ``gap'' $|a-b|$, we see that $G$ contains at least one {\it cyclic} loop. By the same token, if $G$ does contains no cyclic loops of length $N$, then it must contain a cyclic loop containing some vertex $i$, such that $i$ is not connected to $i+1$ or $i-1$ by an edge. To summarize, we have

\begin{lem}\label{cyclic-loop}
Let $G$ be a valence $2m$ graph with no crossings. Then $G$ contains a cyclic loop. If $G$ contains no cyclic loop of length $N$, then it contains a cyclic loop containing some vertex $i$, such that $i$ is not connected to $i+1$ or $i-1$ by an edge.
\end{lem}

We now describe the $\hat\fg$-action on $Re^f$ graphically. Recall that under the identification $\Gamma(X,\cO(1))\equiv\wedge^2\C^N$, $x_{pq}\equiv x_p\wedge x_q$, the Lie algebra $\hat\fg=\fs\fl_N\oplus \C$ action can be represented by differential operators of the form $x_j\partial_i$, $i\neq j$, and $\half x_i\partial_i+1$, $i=1,..,N$, where
\begin{equation}\label{action-xpq}
x_j\partial_i x_{pq}=x_j\partial_i(x_p\wedge x_q)=\delta_{pi}x_{jq}+\delta_{qi}x_{pj}.
\end{equation}
(Note that $x_{pp}=0$ by definition.) This action extends to $R$ and to $Re^f$ by derivations, since the ideal generated by the Pl\"ucker relations is $\hat\fg$-stable. We can use \eqref{action-xpq} to compute graphically the action of $x_j\partial_i$ on any given graph. Figure 2 gives an example, where in each graph, we show only those edges affected by $x_j\partial_i$; every other edge not affected by this operator is the same in all four graphs.

\begin{figure}[htb]
\vskip-1in
\hskip-2in\includegraphics[height=7in,width=5.5in,angle=0]{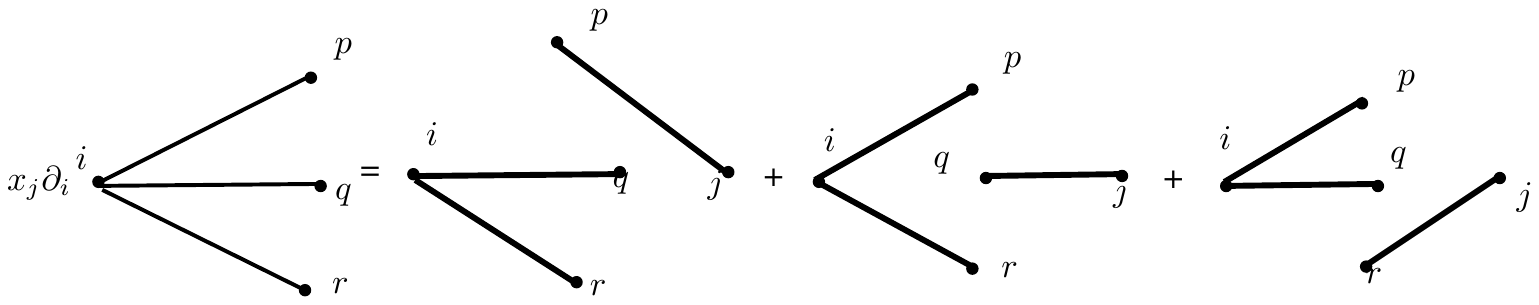}
\vskip-5in\caption{$\hat\fg$ action on graphs.}
\end{figure}

In the following, for $a\in Re^f$, we will write
$$
a\equiv0
$$
if $a\in \hat\fg(Re^f)$. Then \eqref{action-xpq} yields

\begin{lem}\label{fef}
Fix $i$ and let $G$ be a valence 2 graph. 
Then
$$
0\equiv(\half x_i\partial_i+1)(Ge^f)=2Ge^f+Gf e^f.
$$
\end{lem}

\begin{lem}
If $G_1$ is a valence 2 graph containing a cyclic loop of length 2, then $G_1e^f\equiv0$.
\end{lem}

\begin{proof} 
Let $G_1$ be a valuence 2 graph containing a cyclic loop of length 2, running through vertices $i,i+1$. Then vertex $i-1$ must be connected to some vertex $k\neq i-1, i,i+1$. Let $G_2$ be the graph obtained from $G_1$ by removing two edges $(i,i+1)$ and $(i-1,k)$, and replacing them with $(i-1,i)$, $(i,k)$ respectively, as in figure 3, while keeping all other edges the same (not shown in figure).

\begin{figure}[htb]
\vskip-1in
\hskip-2in\includegraphics[height=7in,width=5.5in,angle=0]{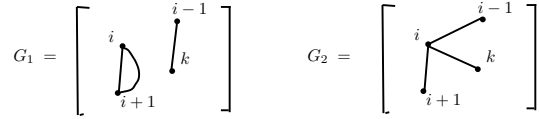}
\vskip-5in\caption{A cyclic loop of length 2, and its modification.}
\end{figure}

We will show that
$$
x_{i+1}\partial_i(G_2e^f)\equiv -G_1e^f.
$$

\newpage
Computing the left side (suppressing all irrelevant edges): 

\begin{figure}[htb]
\vskip-1in
\hskip-2in\includegraphics[height=8.5in,width=7in,angle=0]{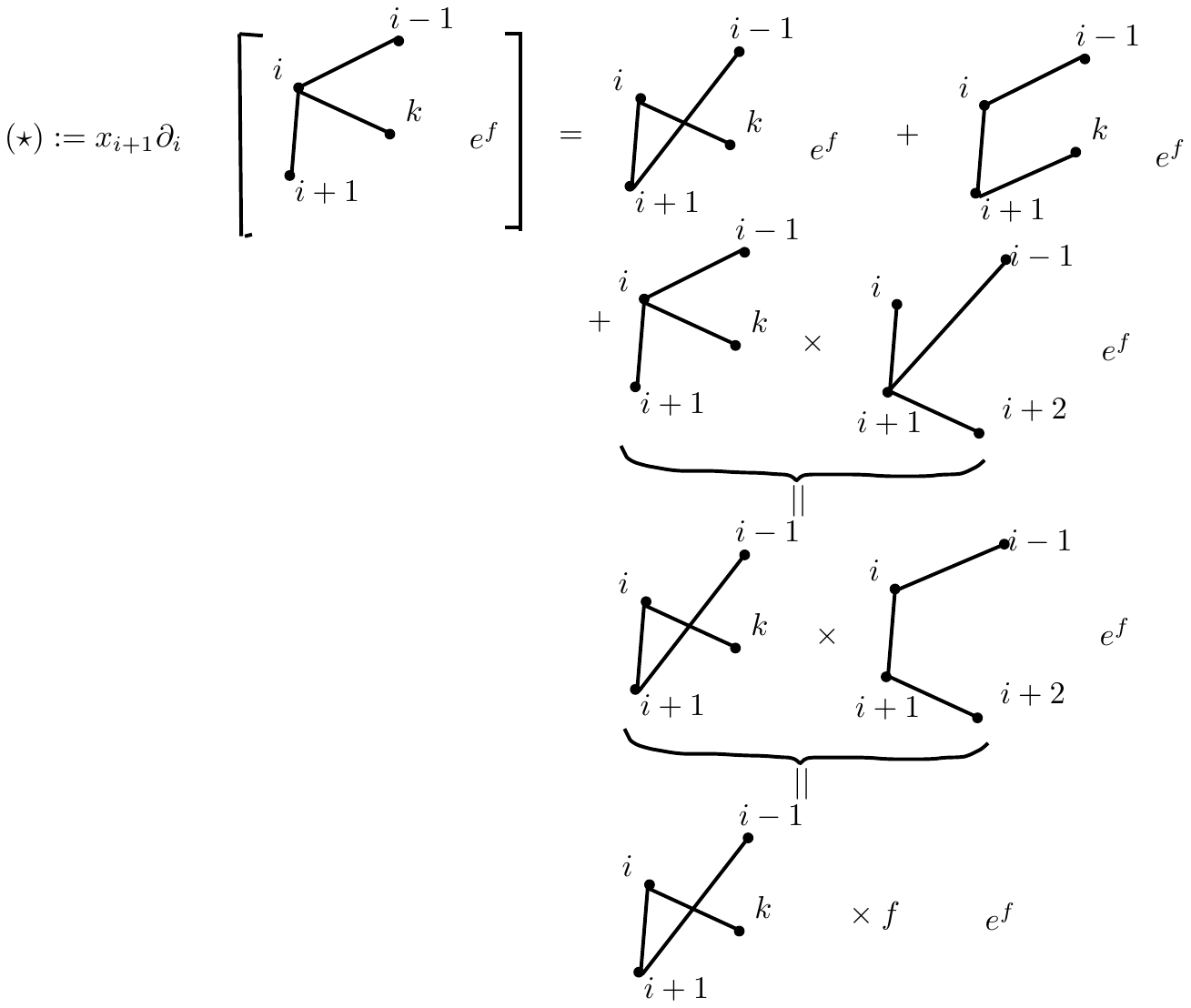}
\vskip-4in\caption{$\fg\fl_n$ action on $G_2e^f$.}
\end{figure}
By exchanging the two edges $(i-1,i+1)$ and $(i-1,i)$ from the two factor graphs in the third term on the right side, and by applying the preceding lemma, we get

\begin{figure}[htb]
\vskip-1in
\hskip-1in\includegraphics[height=7in,width=7in,angle=0]{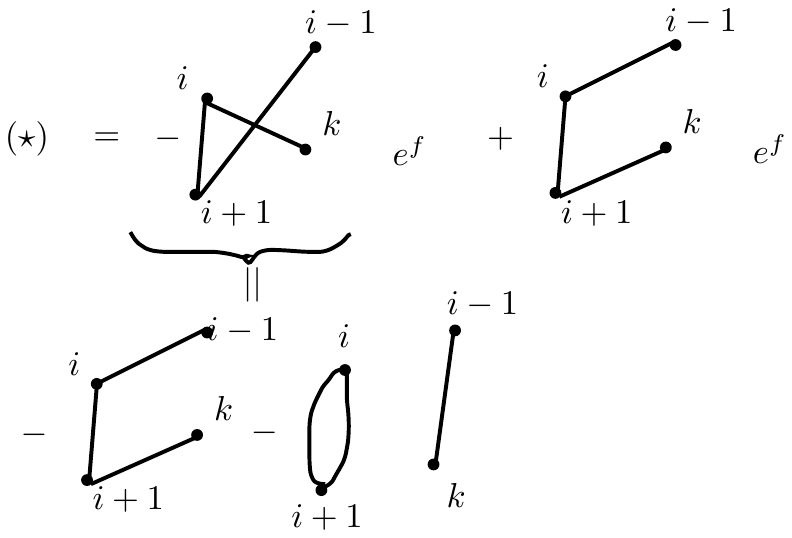}
\vskip-4.5in\caption{Applying a Pl\"ucker operation.}
\end{figure}
As shown in figure 5, now applying a Pl\"ucker operation to remove the crossing in the first term on the right side yields
$$
(\star)=x_{i+1}\partial_i(G_2e^f)\equiv -G_1 e^f.
$$
This completes the proof.
\end{proof}

A parallel calculation, where the cyclic loop of length 2 is replaced one of length $s$, gives

\begin{lem}
If $G_1$ is a valence 2 graph containing a cyclic loop of length $s\leq N-1$, then 
$$
0\equiv x_{i+s-1}\partial_i(G_2e^f)\equiv - G_1e^f+ G_3e^f
$$
for some graphs $G_2,G_3$, where $G_3$ has loops of lengths at most $s-1$.
\end{lem}

The two preceding lemmas and Lemma \ref{fef} imply

\begin{cor}\label{G1ef=ef}
If $G_1$ is a  degree $N$ monomial (graph), then $G_1e^f\equiv \text{const.}~e^f$.
\end{cor}


\begin{lem}\label{hfef=hef}
For any homogeneous $h\in R$, we have $hfe^f\equiv\text{const.~} he^f$.
\end{lem}

\begin{proof}
We have
$$
0\equiv(\half x_i\partial_i+1)(he^f)=he^f+\text{const.~}he^f+hf e^f
$$
hence the claim follows.
\end{proof}

Corollary \ref{G1ef=ef}, Lemma \ref{hfef=hef} and the next lemma imply the statement (*), completing the proof of Theorem \ref{rank1-points}.

\begin{lem}
Let $g\in R$ be homogeneous of degree $mN>0$. Then there exists a homogeneous $h\in R$ of degree $(m-1)N$ such that
$$
ge^f\equiv hfe^f.
$$
\end{lem}

\begin{proof}
Without loss of generality, we can assume that $g$ is a $T$-invariant monomial. Let $F,G$ be the graphs representing the monomials $f,g$ respectively. Again, by Corollary \ref{no-crossings} we may as well assume that $G$ has valence $2m$ but no crossings, hence $G$ contains at least one cyclic loop by Lemma \ref{cyclic-loop}. Introduce the $D$-value of a graph $G'$, $D(G')\leq N$, defined to be the number of edges in $F$ but not in $G'$. Since $G$ contains a cyclic loop, we have $D(G)<N$. If $D(G)=0$, then $G$ contains $F$ as a subgraph and we are done. So, we can assume that $0<D(G)<N$. It suffices to show that 
\eq{GG}{Ge^f\equiv\sum_i c_iG_i e^f}
where $c_i\in\Z$ and the $G_i$ are graphs with no crossings and with $D(G_i)<D(G)$. For then repeatedly applying this argument to the right side of \eqref{GG} yields a sum of graphs $G'$ with $D(G')=0$, and we are done.

By Lemma \ref{cyclic-loop} again, $G$ contains a cyclic loop of length $s$, containing some vertex $i$ such that $i$ is not connected to $i-1$ or $i+1$ by an edge. We will focus on the $i-1$ case with $s=2$. We omit a parallel (but messier) argument for the remaining case. As before (cf. Lemma \ref{cyclic-loop}), by valence consideration, $i-1$ must be connected to at least one other vertex $k\neq i-1,i,i+1$. Moreover, if $k=i-2$ then $G$ contains at least two edges connecting $i-2$ and $i-1$. Let $G_1$ be the graph obtained from $G$ by removing two edges $(i,i+1)$ and $(i-1,k)$, and replacing them with $(i-1,i)$, $(i,k)$ respectively, as in the next figure, while keeping all other edges the same (not shown in figure).

\begin{figure}[htb]
\vskip-1in
\hskip-2in\includegraphics[height=7in,width=5.5in,angle=0]{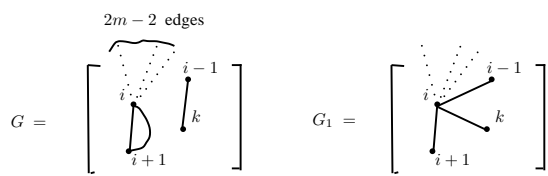}
\vskip-5in\caption{The graph $G$, and its modification $G_1$.}
\end{figure}

\newpage
We get
\begin{figure}[htb]
\vskip-1in
\hskip-2in\includegraphics[height=8.5in,width=7in,angle=0]{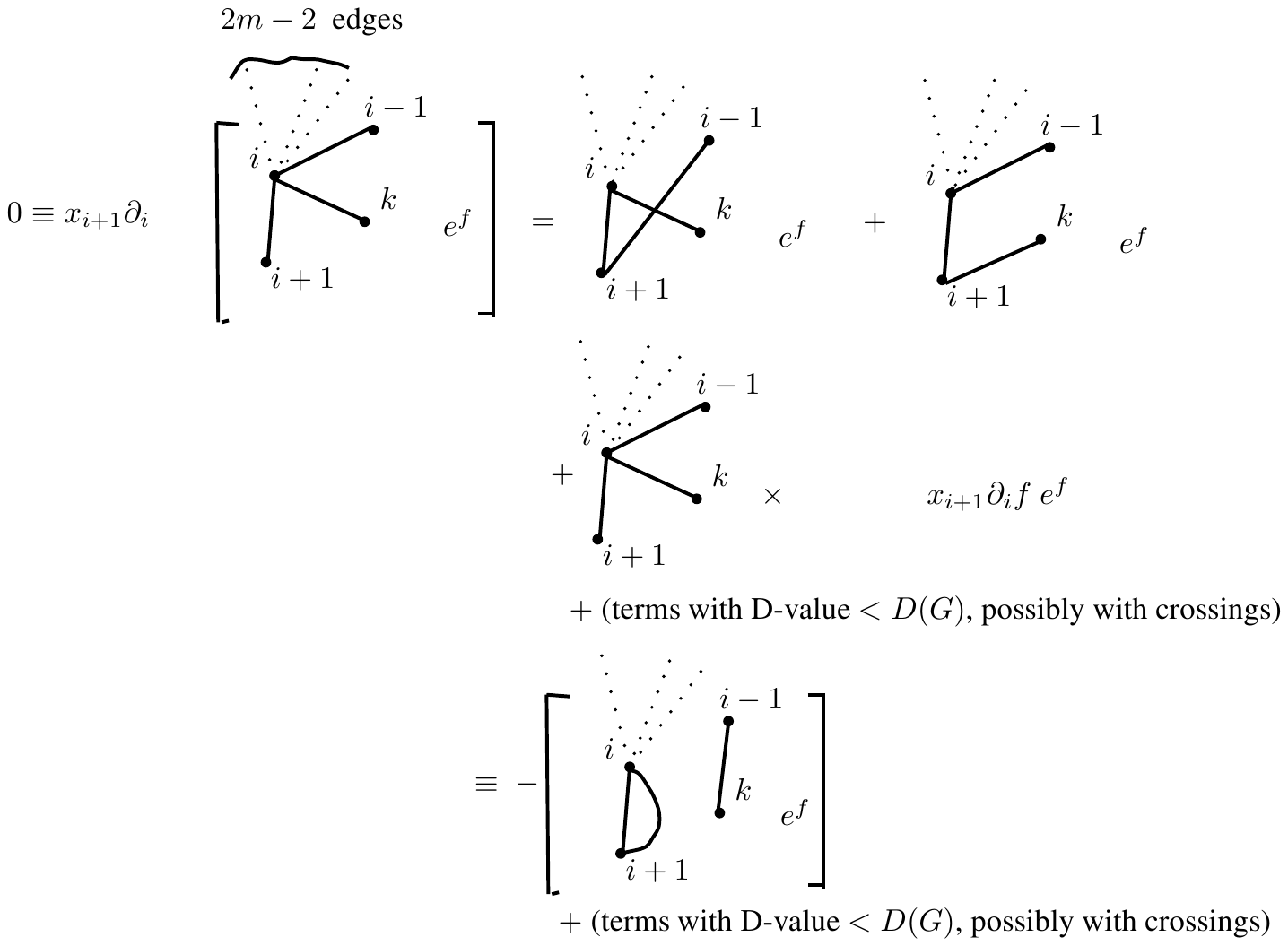}
\vskip-4in\caption{$\fg\fl_n$ action on $G_1e^f$.}
\end{figure}

Note that the terms with $D$-value $<D(G)$ above come from applying $x_{i+1}\partial_i$ to the $2m-2$ edges (dotted lines connected to $i$) in $G_1$. The $D$-value decreases by at least one relative to $D(G)$, because of the presence of the edge $(i-1,i)$ in each of those terms. Now, observe that a Pl\"ucker operation does not increase the $D$-value of a graph, and so we can apply such an operation to remove each crossing on the right side while keeping the $D$-value of each term $<D(G)$. Since the first term on the right side is $-Ge^f$, it follows that $Ge^f$ is equivalent to terms with $D$-value $<D(G)$ with no crossings, as desired.
\end{proof}

\section{A geometric proof}

We next give a geometric proof of Theorem \ref{rank1-points}.

Write $\C^N = \C e_1\oplus\cdots\oplus \C e_N$. We have $X\inj \P(\bigwedge^2 \C^N)$, and $X$ is identified with the space of decomposible tensors $v\wedge w$ upto scale. Write $v=\sum a_ie_i,\ w=\sum b_ie_i$. Then $x_{ij}(v\wedge w) = a_ib_j-a_jb_i$. We identify
\eq{}{X-\{x_{12}=0\} \cong Hom(\C e_1\oplus \C e_2, \bigoplus_{3\le i\le N}\C e_i)
}
in the standard way, which amounts to taking decomposible elements $v\wedge w$ with $v=e_1+\sum_{i\ge 3}a_ie_i$ and $w=e_2+\sum_{i\ge 3}b_ie_i$. We want to compute
\eq{}{H^{2N-4}(X-\{x_{12}\cdots x_{N-1,N}x_{N,1}=0\})
}
We have
\ml{3}{X-\{x_{12}\cdots x_{N-1,N}x_{N,1}=0\} \cong \\
\Spec \C[a_3,b_3,\dotsc,a_N,b_N,\frac{1}{a_3},\frac{1}{a_3b_4-a_4b_3},\dotsc,\frac{1}{a_{N-1}b_N-a_Nb_{N-1}},\frac{1}{b_N}]
}
Define $V_3=\Spec \C[a_3,b_3,1/a_3] \cong \G_m\times \G_a$ where I write $\G_m=\C^\times$ for the multiplicative group and $\G_a=\C$ for the additive group. More generally, for $p\ge 4$
\eq{}{V_p := \Spec \C[a_3,b_3,\dotsc,a_p,b_p,\frac{1}{a_3},\frac{1}{a_3b_4-a_4b_3},\dotsc,\frac{1}{a_{p-1}b_p-a_pb_{p-1}}]
}

Let $\cG:= \G_m \rtimes \G_a$ be the group of affine transformations $x\mapsto ux+v$. Let $\pi_p: V_p \to V_{p-1}$ be the evident projection. We have
\ml{}{\pi_p^{-1}(\alpha_3,\beta_3,\dotsc,\alpha_{p-1},\beta_{p-1}) = \\
\{ (\alpha_3,\beta_3,\dotsc,\alpha_{p},\beta_{p})\ |\ \det\begin{pmatrix}\alpha_{p-1} & \alpha_p \\\beta_{p-1} & \beta_p\end{pmatrix} \neq 0\} 
}
The action of $\cG$ on $V_p/V_{p-1}$ given by 
\ml{}{(u,v)\cdot (\ldots,\alpha_{p-1},\beta_{p-1},\alpha_p,\beta_p) = \\
(\ldots, \alpha_{p-1},\beta_{p-1},u\alpha_p+v\alpha_{p-1},u\beta_p+v\beta_{p-1})
}
makes $V_p$ a principal $\cG$-bundle over $V_{p-1}$. But any such $\cG$-bundle is split, because $V_{p-1}$ affine implies $H^1(V_{p-1},\G_a)=(0)$, and $H^1(V_{p-1},\G_m)=(0)$ implies the set of $\cG$-bundles on $V_{p-1}$ which split when pushed out to $\G_a$ has one element. Thus $V_p\cong V_{p-1}\times \G_m\times \G_a$ as a variety. We conclude
\eq{7}{V_p \cong \G_m^{p-2}\times \G_a^{p-2}. 
}
In particular, 
\eq{8}{H^i(V_p, \Z)=(0),\ i\ge p-1.
} 

Next define $W_p \inj V_p$ to be the closed subvariety defined by $b_p=0$. One gets a diagram of bundles
\eq{9}{\begin{CD}W_p @>>> V_p \\
@VV \G_m V @VV\cG V \\
V_{p-1}-W_{p-1} @>>> V_{p-1}.
\end{CD}
}
These are open subvarieties of affine space, so the Picard groups vanish and we have 
\eq{10}{W_p \cong (V_{p-1}-W_{p-1})\times \G_m.
}
We prove by induction on $p\ge 3$ that 
\eq{11}{H^i(V_p-W_p,\Z)=(0);\ \  i\ge 2p-3;\quad H^{2p-4}(V_p-W_p) = \Z.
}
For $p=3$, the assertions are $H^i(\G_m^2)=(0), i\ge 3$ and $H^2(\G_m^2) = \Z$, both of which are true. For $p>3$ we have the Gysin sequence
\eq{12}{H^{i}(V_p) \to H^{i}(V_p-W_p)\to H^{i-1}(W_p) \to H^{i+1}(V_p)
}
Since $2p-4\ge p-1$ in our case, we see from \eqref{8}, \eqref{10}, and \eqref{12} that 
\ml{13}{H^{i}(V_p-W_p)\cong H^{i-1}(W_p)\cong H^{i-1}((V_{p-1}-W_{p-1})\times \G_m)\cong \\
H^{i-1}((V_{p-1}-W_{p-1})\oplus H^{i-2}(V_{p-1}-W_{p-1})
}
By induction we get the desired vanishing for $i\ge 2p-3$. For $i=2p-4$ the same argument yields 
\eq{14}{H^{2p-4}(V_p-W_p) \cong H^{2p-6}(V_{p-1}-W_{p-1})\cong \Z.
}
Again we conclude by induction. 

In the case $p=N$ we get from \eqref{14} that $H^{2N-4}(V_N-W_N) \cong \Z$ as desired, completing the geometric proof of Theorem \ref{rank1-points}.

\noindent {\bf Acknowledgements.} B.H.L. is partially supported by NSF FRG grant DMS-0854965, and S.T.Y. by NSF FRG grant DMS-0804454. A.H. has benefited greatly from discussions with Marcel B\"okstedt and Shenghao Sun, and part of the work was done during his visit to the Tsinghua Mathematical Sciences Center. S.B. would also like to acknowledge support from the Tsinghua Mathematical Sciences Center and from the Tata Institute for Fundamental Research. His role in the project grew out of conversations he had at these institutions in the fall and winter of 2011-2012. 

{\it Note added:} It has recently been shown that Conjecture 1.3 is in
fact a special case of a much more general formula for solution ranks
of tautological systems. This result will appear in a forthcoming
paper \cite{HLZ}.


\noindent\address {\SMALL S. Bloch, 5765 S. Blackstone Ave., Chicago IL 60637.  \\ spencer\_bloch@yahoo.com.}
\vskip-.15in

\noindent\address {\SMALL A. Huang, Department of Mathematics, Harvard University, Cambridge MA 02138. \\ anhuang@math.harvard.edu.}
\vskip-.15in

\noindent\address {\SMALL B.H. Lian, Department of Mathematics, Brandeis University, Waltham MA 02454.\\ lian@brandeis.edu.}
\vskip-.15in

\noindent\address {\SMALL V. Srinivas, School of Mathematics, Tata Institute for Fundamental Research,  Homi Bhabha Road, Mumbai 400005, India.\\ srinivas@tifr.res.in}
\vskip-.15in

\noindent\address  {\SMALL S-T. Yau, Department of Mathematics, Harvard University, Cambridge MA 02138. \\ yau@math.harvard.edu.}

\end{document}